\documentclass[12pt]{article}
\usepackage{amsmath,amssymb,amsfonts}
\usepackage{amsthm}%
\usepackage{epsfig}
\usepackage{enumerate}

\usepackage{wasysym}

\usepackage{tikz}

{\theoremstyle{plain}%
 \newtheorem{theorem}{Theorem}

 \newtheorem{lemma}{Lemma}%
}
{\theoremstyle{remark}

}
{\theoremstyle{definition}
\newtheorem{definition}{Definition}
\newtheorem{example}{Example}
}

\begin{document}

\begin{center}
{\large A Schur--Weyl duality analogue based on a commutative bilinear operation}

 \ 

{\sc John M. Campbell}
\end{center}

\begin{abstract}
 Schur--Weyl duality concerns the actions of $\text{GL}_{n}(\mathbb{C})$ and $S_{k}$ on tensor powers of the form $V^{\otimes k}$ for an 
 $n$-dimensional vector space $V$. There are rich histories within representation theory, combinatorics, and statistical mechanics involving the study and 
 use of diagram algebras, which arise through the restriction of the action of $\text{GL}_{n}(\mathbb{C})$ to subgroups of $\text{GL}_{n}(\mathbb{C})$. 
 This leads us to consider further variants of Schur--Weyl duality, with the use of variants of the tensor space $V^{\otimes k}$. Instead of taking repeated 
 tensor products of $V$, we make use of a freest commutative bilinear operation in place of $\otimes$, and this is motivated by an associated 
 invariance property given by the action of $S_{k}$. By then taking the centralizer algebra with respect to the action of the group 
 of permutation matrices in $\text{GL}_{n}(\mathbb{C})$, this gives rise to a diagram-like algebra spanned by a new class of combinatorial objects. We 
 construct orbit-type bases for the centralizer algebras introduced in this paper, and we apply these bases to prove a combinatorial formula for the 
 dimensions of our centralizer algebras. 
\end{abstract}



\section{Introduction}\label{sectionIntro}
 Let $V \cong \mathbb{C}^{n}$ denote an $n$-dimensional complex vector space. We then let the symmetric group $S_{k}$ act on $V^{\otimes k}$ 
 according to the \emph{place-permutation} action such that 
\begin{equation}\label{placeaction}
 \sigma \left( \text{{\bf w}}_{i_{1}} \otimes \text{{\bf w}}_{i_{2}} \otimes \cdots \otimes \text{{\bf w}}_{i_{k}} \right) 
 = \text{{\bf w}}_{\sigma^{-1}(i_{1})} \otimes \text{{\bf w}}_{\sigma^{-1}(i_{2})} \otimes \cdots \otimes 
 \text{{\bf w}}_{\sigma^{-1}(i_{k})} 
\end{equation}
 for a permutation $\sigma \in S_{k}$ and for elements $\text{{\bf w}}_{i_{1}}$, $\text{{\bf w}}_{i_{2}}$, $\ldots$, $\text{{\bf w}}_{i_{k}}$ in $V$, 
 for distinct indices $i_{1}, i_{2}, \ldots, i_{k} \in \{ 1, 2, \ldots, k \}$. For a given basis $\{ \text{{\bf v}}_{1}, \text{{\bf v}}_{2}, \ldots, 
 \text{{\bf v}}_{n} \}$ of $V$, we let the general linear group $\text{GL}_{n}(\mathbb{C})$ act on $V$ so that
\begin{equation}\label{20240602728PM1A}
 g \text{{\bf v}}_{i} = \sum_{j = 1}^{n} g_{j, i} \text{{\bf v}}_{j}, 
\end{equation}
 with this action being extended diagonally so that 
\begin{equation}\label{gdiag}
 g\left( \text{{\bf v}}_{i_{1}} \otimes \text{{\bf v}}_{i_{2}} \otimes \cdots \otimes \text{{\bf v}}_{i_{k}} \right) 
 = g \text{{\bf v}}_{i_{1}} \otimes g \text{{\bf v}}_{i_{2}} \otimes \cdots \otimes g \text{{\bf v}}_{i_{k}}. 
\end{equation}
 For a vector space $W$, we let $\text{End}(W)$ denote the the algebra of linear endomorphisms on $W$, and, for a group or group algebra $A$ acting 
 on $W$, we let $\text{End}_{A}(W)$ denote the subalgebra of $\text{End}(W)$ consisting of morphisms that are centralized with respect to the action 
 of $A$. Extending the group actions indicated in \eqref{placeaction} and \eqref{gdiag}, respectively, to actions by the group algebras $\mathbb{C} 
 S_k$ and $\mathbb{C} \text{GL}_{n}(\mathbb{C})$, \emph{Schur--Weyl duality} (see \cite{Schur1927}
 and \cite{Weyl1939}) may be understood to refer to these group algebra actions 
 commuting and to how there exist surjective homomorphisms from $\mathbb{C} \text{GL}_{n}(\mathbb{C})$ to $\text{End}_{\mathbb{C}S_k}\left( 
 V^{\otimes k} \right)$ and from $\mathbb{C}S_{k}$ to $\text{End}_{\mathbb{C} \text{GL}_{n}(\mathbb{C}) }\left( V^{\otimes k} \right)$
 (given, respectively, by the natural representations of 
 $\mathbb{C} \text{GL}_{n}(\mathbb{C})$ and of $\mathbb{C}S_{k}$), 
 with injectivity for $n \geq k$; see the work 
 of Cruz \cite{Cruz2019} and of Dipper et al.\ \cite{DipperDotyHu2008}, for example, 
 and references therein. Due to how each action generates the centralizer algebra of the other action, we obtain the decomposition 
\begin{equation}\label{SchurWeyldecomposition}
 V^{\otimes k} \cong \bigoplus_{\lambda} L^{\lambda}_{\text{GL}_{n}(\mathbb{C})} \otimes S_{k}^{\lambda} 
\end{equation}
 of $V^{\otimes k}$ as a bimodule, under the specified actions, where $L^{\lambda}_{\text{GL}_{n}(\mathbb{C})}$ denotes the irreducible $\mathbb{C} 
 \text{GL}_{n}(\mathbb{C})$-module corresponding to an integer partition $\lambda$, and where $S_{k}^{\lambda}$ denotes the Specht module 
 corresponding to the same partition $\lambda$, with reference to the work of Halverson and Ram
 \cite{HalversonRam2005} on partition algebras and Schur--Weyl duality. 
 Since the decomposition in \eqref{SchurWeyldecomposition} is multiplicity-free, and since the same index set of partitions is used for the irreducible 
 modules for both the general linear and symmetric groups, this provides a fundamental connection within the field of representation theory and within related 
 areas of combinatorics: Informally, this is due to how the formulation of Schur--Weyl duality in \eqref{SchurWeyldecomposition} 
 gives us how the 
 representations of $\text{GL}_{n}(\mathbb{C})$ provide information on the representations of $S_{k}$ and vice-versa. The representation theory of 
 the symmetric group may be seen as being of basic importance within the representation theory of groups and within many associated disciplines of 
 combinatorics, with reference to the classic text by Sagan \cite{Sagan2001}, and hence the significance as to how the decomposition in 
 \eqref{SchurWeyldecomposition} relates the representations of $S_{k}$ with matrix group representations. 

 The foregoing considerations motivate the generalization of Schur--Weyl duality with the use of variants of actions as in the general linear group action 
 defined via \eqref{gdiag}. In this regard, the notion of a \emph{diagram algebra} naturally arises both within the field of statistical mechanics and through 
 the restriction of the $\text{GL}_{n}(\mathbb{C})$-action in \eqref{gdiag}, by restricting $g$ to a specified matrix subgroup containing the permutation 
 matrices in $\text{GL}_{n}(\mathbb{C})$. In particular, by rewriting $\text{End}_{\mathbb{C}\text{GL}_{n}(\mathbb{C})}\left( V^{\otimes k} \right)$ 
 as $\text{End}_{\text{GL}_{n}(\mathbb{C})}\left( V^{\otimes k} \right)$, and by identifying $S_{n}$ with the group of permutation matrices in 
 $\text{GL}_{n}(\mathbb{C})$, we restrict the action of $\text{GL}_{n}(\mathbb{C})$ to $S_{n}$, so as to obtain the isomorphic equivalence 
\begin{equation}\label{EndSnVtensor}
 \text{End}_{S_{n}}\left( V^{\otimes k} \right) \cong \mathbb{C}A_k(n) 
\end{equation}
 for $2k \leq n$, where the right-hand side of \eqref{EndSnVtensor} is the \emph{partition algebra} of order $k$, as defined by Halverson and Ram 
 \cite{HalversonRam2005} with the use of combinatorial objects known as \emph{partition diagrams}. We instead consider an analogue of 
 \eqref{EndSnVtensor} with the use of a freest \emph{commutative} bilinear operation in place of $\otimes$. 

\section{A Schur--Weyl duality analogue based on a commutative bilinear operation}
 For the purposes of this paper, we let vector spaces be over $\mathbb{C}$. For two such spaces $V$ and $W$, we recall the construction of $V \otimes 
 W$ as a quotient space and consider a commutative variant of this construction. 
 Letting $L$ denote the Cartesian product $V \times W$, 
 we then let $R_{1}$ denote the subspace of $L$ spanned 
 by expressions of the forms $(\text{{\bf v}}_{1} + \text{{\bf v}}_{2}, \text{{\bf w}}) - (\text{{\bf v}}_{1}, \text{{\bf w}}) - (\text{{\bf v}}_{2}, \text{{\bf w}})$ and $(\text{{\bf v}}, \text{{\bf w}}_{1} + \text{{\bf w}}_{2}) - (\text{{\bf v}}, \text{{\bf w}}_{1}) - (\text{{\bf v}}, \text{{\bf w}}_{2})$ and $(s \text{{\bf v}}, \text{{\bf w}}) - s(\text{{\bf v}}, \text{{\bf w}})$ and $(\text{{\bf v}}, s \text{{\bf w}}) - s (\text{{\bf v}}, \text{{\bf w}})$, for $\text{{\bf v}}, \text{{\bf v}}_{1}, \text{{\bf v}}_{2} \in V$ and $\text{{\bf w}}, \text{{\bf w}}_{1}, \text{{\bf w}}_{2} \in W$
 and $s \in \mathbb{C}$. The tensor space $V \otimes W$ may then be identified
 with the quotient space $L/R_{1}$, and the equivalence class containing the tuple $(\text{{\bf v}}, \text{{\bf w}})$
 may be written as $\text{{\bf v}} \otimes \text{{\bf w}}$. 
 We then set $R_{2}$ as the subspace of $L$ spanned by the union of $R_{1}$
 and the set of all 
 expressions of the form $(\text{{\bf v}}, \text{{\bf w}}) - (\text{{\bf w}}, \text{{\bf v}})$, 
 again for $\text{{\bf v}} \in V$ and for $\text{{\bf w}} \in W$. 
 We then rewrite the quotient space $L/R_{2}$ as 
 $ V \boxtimes W$, 
 and we rewrite the equivalence class containing $(\text{{\bf v}}, \text{{\bf w}})$ as 
 $\text{{\bf v}} \boxtimes \text{{\bf w}}$. 
 The operation $\boxtimes$ may be seen as a freest commutative bilinear operation, by analogy 
 with the universal property that may be used to define $\otimes$. 

 Again setting $V \cong \mathbb{C}^{n}$, we write $V^{\boxtimes k}$ in place of the $k$-fold $\boxtimes$-product 
\begin{equation}\label{boxpower}
 V^{\boxtimes k} = \underbrace{V \boxtimes V \boxtimes \cdots \boxtimes V}_{k}. 
\end{equation}
 By analogy with \eqref{placeaction}, we let $S_{k}$
 act on $ V^{\boxtimes k}$ by setting 
\begin{equation}\label{Skactbox}
 \sigma \left( \text{{\bf w}}_{i_{1}} \boxtimes 
 \text{{\bf w}}_{i_{2}} \boxtimes \cdots \boxtimes \text{{\bf w}}_{i_{k}} \right) 
 = \text{{\bf w}}_{\sigma^{-1}(i_{1})} \boxtimes 
 \text{{\bf w}}_{\sigma^{-1}(i_{2})} \boxtimes \cdots \boxtimes 
 \text{{\bf w}}_{\sigma^{-1}(i_{k})} 
\end{equation}
 for $\sigma \in S_{k}$ and for $\text{{\bf w}}_{i_{1}}$, $\text{{\bf w}}_{i_{2}}$, $\ldots$, $\text{{\bf w}}_{i_{k}}$ as above. The action defined in 
 \eqref{Skactbox} is inspired by an invariance property that is 
 described below and that is related to the classical formulation of Schur--Weyl duality 
 reviewed in Section \ref{sectionIntro}. 
 By again fixing $\{ \text{{\bf v}}_{1}, \text{{\bf v}}_{2}, \ldots, \text{{\bf v}}_{n} \}$ 
 as a basis of $V$, we find that 
 the set of all expressions of the form 
 $ \text{{\bf v}}_{i_{1}} \boxtimes \text{{\bf v}}_{i_{2}} \boxtimes 
 \cdots \boxtimes \text{{\bf v}}_{i_{k}}$, up to equivalence 
 by permutations of the positions of the factors in such $ \boxtimes$-products, 
 and for indices $i_{1}$, $i_{2}$, $\ldots$, $i_{k}$ that are not necessarily distinct, 
 forms a basis of \eqref{boxpower}. 
 This gives us that the dimension of $V^{\boxtimes k}$ 
 is equal to the number of integer partitions of length $k$
 with entries from $\{ 1, 2, \ldots, n \}$. 

 With respect to the action defined in \eqref{Skactbox}, 
 we obtain the centralizer algebra 
\begin{equation}\label{20100000002470758282717075777A7KM1A}
 \text{End}_{S_{k}}\left( V^{\boxtimes k} \right) = 
 \left\{ f \in \text{End}\left( V^{\boxtimes k} \right) : 
 \forall \sigma \in S_{k} \forall t \in V^{\boxtimes k} \ 
 f \sigma t = \sigma f t \right\}. 
\end{equation}
 For a given basis element $ \text{{\bf b}} = \text{{\bf v}}_{i_{1}} \boxtimes \text{{\bf v}}_{i_{2}} \boxtimes \cdots \boxtimes \text{{\bf v}}_{i_{k}}$ 
 of $V^{\boxtimes k}$, we have that $\sigma \text{{\bf b}} = \text{{\bf b}}$ for all permutations $\sigma \in S_{k}$. So, for any endomorphism $f\colon 
 V^{\boxtimes k} \to V^{\boxtimes k}$ by letting $ f(\text{{\bf c}}_{i}) = \sum_{j \in I} c_{i, j} \text{{\bf c}}_{j} $ for an index set $I$ and scalars 
 $c_{i}$ and a basis $\{ \text{{\bf c}}_{i} \}_{i \in I}$ of $V^{\boxtimes k}$, we find that $ \sigma f(\text{{\bf c}}_{i}) = \sum_{j \in I} c_{i, j} \left( 
 \sigma \text{{\bf c}}_{j} \right) = f(\text{{\bf c}}_{i}) $ and that $ f\left( \sigma \, \text{{\bf c}}_{i} \right) = f\left( \text{{\bf c}}_{i} \right)$, 
 which gives us the invariance property 
\begin{equation}\label{invariance}
 \text{End}_{S_{k}}\left( V^{\boxtimes k} \right) = \text{End}\left( V^{\boxtimes k} \right). 
\end{equation}

 By analogy with \eqref{gdiag}, we set $ g\left( \text{{\bf v}}_{i_{1}} \boxtimes \text{{\bf v}}_{i_{2}} \boxtimes \cdots \boxtimes 
 \text{{\bf v}}_{i_{k}} \right) = g \text{{\bf v}}_{i_{1}} \boxtimes g \text{{\bf v}}_{i_{2}} \boxtimes \cdots \boxtimes g \text{{\bf v}}_{i_{k}}$, 
 again for $g \in \text{GL}_{n}(\mathbb{C})$. For a permutation $\sigma$ of order $n$, we write $\sigma\colon \{ 1, 2, \ldots, n \} \to \{ 1, 2, \ldots, n \}$ 
 as a bijection, 
 and we identify $\sigma$ with the permutation matrix $g \in \text{GL}_{n}(\mathbb{C})$ 
 such that the $(\sigma(i), i)$-entry of $g$ equals $1$, 
 with $0$-entries elsewhere. So, the permutation $\sigma$ 
 acts on $V$ so that $\sigma \text{{\bf v}}_{i} = \sum_{i=1}^{n} g_{\sigma(i), i} \text{{\bf v}}_{\sigma(i)}$ 
 according to the notation in \eqref{20240602728PM1A}. 
 This action is extended so that 
\begin{equation}\label{extendpermutation}
 \sigma\left( \text{{\bf v}}_{i_{1}} \boxtimes \text{{\bf v}}_{i_{2}} \boxtimes \cdots \boxtimes \text{{\bf v}}_{i_{k}} \right) 
 = \text{{\bf v}}_{\sigma(i_{1})} \boxtimes \text{{\bf v}}_{\sigma(i_{2})} 
 \boxtimes \cdots \boxtimes \text{{\bf v}}_{\sigma(i_{k})}, 
\end{equation}
 with \eqref{extendpermutation} to 
 be extended linearly. 
 Observe how the action in \eqref{extendpermutation} 
 of $S_{n}$, as a subgroup of $\text{GL}_{n}(\mathbb{C})$, on $V^{\boxtimes k}$, 
 is inequivalent to the action of $S_{k}$ on $V^{\boxtimes k}$ indicated in 
 \eqref{Skactbox}. This 
 and the invariance property indicated in \eqref{invariance}
 lead us to investigate the centralizer algebra
 $\text{End}_{S_{n}}\left( V^{\boxtimes k} \right)$, 
 as opposed to $\text{End}_{S_{k}}\left( V^{\boxtimes k} \right)$. 

 The key to the determination of a combinatorial interpretation 
 for basis elements of either side of the isomorphic equivalence in 
 \eqref{EndSnVtensor}
 is given by a 1994 result attributed, as in the work of Bloss \cite{Bloss2005}, 
 to the work of Jones \cite{Jones1994} 
 on the Potts model and the symmetric group. 
 Our derivation, as below, of an analogue of Jones' characterization of elements of the centralizer algebra in 
 \eqref{EndSnVtensor} follows a similar approach as in 
 the presentation by Bloss \cite{Bloss2005} on Jones' characterization. 

 An \emph{integer partition} is a finite tuple of non-decreasing positive integers. For an integer partition $\lambda$, we let the length of the tuple $ 
 \lambda$ be denoted as $\ell(\lambda)$, and we let the entries of $\lambda$ be denoted in such a way so that $\lambda = (\lambda_{1}, 
 \lambda_{2}, \ldots, \lambda_{\ell(\lambda)})$. We let $\lambda^{\text{T}}$ denote the transpose of $\lambda$, which is equal to the integer 
 partition obtained by reflecting the Ferrers diagram of $\lambda$ about the main diagonal, with reference to 
 standard texts on integer partitions, as in the work of Andrews 
 \cite{Andrews1998}. We let $\mathcal{P}$ denote the set of all integer partitions. For a tuple $t$ of positive integers, we let $\text{sort}(t)$ 
 denote the integer partition obtained by sorting the entries of $t$ in nondecreasing order. For an integer partition $\lambda$, it is standard to let 
 $m_{j}(\lambda)$ denote the multiplicity of $j$ in $\lambda$. 

 For an integer partition $\lambda$ such  
 that $\ell(\lambda) = k$ and  such that each entry of $\lambda$ is less than or equal to $n$, 
 we set
 $ \text{{\bf v}}_{\lambda} 
 = \text{{\bf v}}_{\lambda_{1}} \boxtimes \text{{\bf v}}_{\lambda_{2}} 
 \boxtimes \cdots \boxtimes \text{{\bf v}}_{\lambda_k}$, 
 again for the specified basis of the $n$-dimensional space $V$. 
 Our definition for $\boxtimes$ is such that the family 
\begin{equation}\label{basisboxpower}
 \left\{ \text{{\bf v}}_{\lambda} : \ell(\lambda) = k, \ell\left(\lambda^{\text{T}}\right) \leq n \right\} 
\end{equation}
 is a basis for $V^{\boxtimes k}$. 
 This can be thought of as providing an analogue of how bases of the Hopf algebra $\textsf{Sym}$ of symmetric functions 
 are indexed by integer partitions, referring to the text of Macdonald \cite{Macdonald1995} for details. 
 If we consider the Hopf algebra $\textsf{NSym}$ of noncommutative symmetric functions 
 introduced by Gelfand et al.\ \cite{GelfandKrobLascouxLeclercRetakhThibon1995} 
 and how it projects onto $\textsf{Sym}$, 
 this motivates our exploration of 
 the relationship between 
 $\text{End}_{S_{n}}\big( V^{\otimes k} \big)$ 
 and a commutative analogue based on $V^{\boxtimes k}$. 

 Let $\lambda$ be an integer partition satisfying the conditions in 
 \eqref{basisboxpower}. 
 We then let $B \in \text{End}\left( V^{\boxtimes k} \right)$, and we write 
\begin{equation}\label{writeBvlambda}
 B\left( \text{{\bf v}}_{\lambda} \right) = 
 \sum_{\substack{ \mu \in \mathcal{P} \\ \ell(\mu) = k \\ \ell\left( \mu^{\text{T}} \right) \leq n }} 
 B^{\lambda}_{\mu} \text{{\bf v}}_{\mu}. 
\end{equation}

\begin{lemma}\label{analogueJones}
 Let $B \in \text{End}\left( V^{\boxtimes k} \right)$.
 Then $B \in \text{End}_{S_{n}}\left( V^{\boxtimes k} \right)$ if and only if 
\begin{equation}\label{29999990929490959390929194PM1A}
 B_{\mu}^{\lambda} = 
 B^{\text{sort}\left( \sigma\left( \lambda_{1} \right), \sigma\left( \lambda_{2} \right), \ldots, \sigma\left( \lambda_{k} 
 \right) \right)}_{ \text{sort}\left( 
 \sigma\left( \mu_{1} \right), 
 \sigma\left( \mu_{2} \right), \ldots, 
 \sigma\left( \mu_{k} \right) \right) } 
\end{equation}
 for all integer partitions $\lambda$ and $\mu$
 such that $\ell(\lambda) = \ell(\mu) = k$
 and $\ell\left( \lambda^{\text{T}} \right) \leq n$
 and $\ell\left( \mu^{\text{T}} \right) \leq n$
 and for all permutations $\sigma \in S_{n}$. 
\end{lemma}

\begin{proof}
 From \eqref{20100000002470758282717075777A7KM1A}, we have that 
 $B \in \text{End}_{S_{n}}\left( V^{\otimes k} \right)$ if and only if 
 $B (\sigma t) = \sigma B (t)$ for each expression $t \in V^{\boxtimes k}$ 
 of the form $\text{{\bf v}}_{\lambda}$ 
 for $\lambda \in \mathcal{P}$
 such that $\ell(\lambda) = k$
 and such that $\ell\left( \lambda^{\text{T}} \right) \leq n$. 
 Applying $\sigma \in S_{n}$ to both sides of \eqref{writeBvlambda}, we obtain 
\begin{align}
\begin{split}
 \sigma B\left( \text{{\bf v}}_{\lambda} \right) 
 & = \sum_{\substack{ \mu \in \mathcal{P} \\ \ell(\mu) = k \\ \ell\left( \mu^{\text{T}} \right) \leq n }} 
 B^{\lambda}_{\mu} 
 \text{{\bf v}}_{\sigma(\mu_{1})} \boxtimes 
 \text{{\bf v}}_{\sigma(\mu_{2})} \boxtimes \cdots \boxtimes 
 \text{{\bf v}}_{\sigma(\mu_{k})} \\ 
 & = \sum_{\substack{ \mu \in \mathcal{P} \\ \ell(\mu) = k \\ \ell\left( \mu^{\text{T}} \right) \leq n }} 
 B^{\lambda}_{\mu} 
 \text{{\bf v}}_{ \text{sort}\left( 
 \sigma\left( \mu_{1} \right), 
 \sigma\left( \mu_{2} \right), \ldots, 
 \sigma\left( \mu_{k} \right) \right) }. 
\end{split}\label{202405301159AM1A}
\end{align}
 Similarly, we have that 
\begin{align*}
 B\left( \sigma \text{{\bf v}}_{\lambda} \right) 
 & = B\left( \sigma \left( \text{{\bf v}}_{\lambda_{1}} \boxtimes 
 \text{{\bf v}}_{\lambda_{2}} \boxtimes \cdots \boxtimes \text{{\bf v}}_{\lambda_{k}} \right) \right) \\ 
 & = B\left( \text{{\bf v}}_{\sigma(\lambda_{1})} 
 \boxtimes \text{{\bf v}}_{\sigma(\lambda_{2})} \boxtimes \cdots \boxtimes
 \text{{\bf v}}_{\sigma(\lambda_{k})} \right) \\
 & = B\left( \text{{\bf v}}_{ \text{sort}\left( 
 \sigma\left( \lambda_{1} \right), \sigma\left( \lambda_{2} \right), \ldots, \sigma\left( \lambda_{k} \right) \right) 
 } \right) \\ 
 & = \sum_{\substack{ \nu \in \mathcal{P} \\ \ell(\nu) = k \\ \ell\left( \nu^{\text{T}} \right) \leq n }} 
 B^{\text{sort}\left( \sigma\left( \lambda_{1} \right), \sigma\left( \lambda_{2} \right), \ldots, \sigma\left( \lambda_{k} 
 \right) \right)}_{\nu} \text{{\bf v}}_{\nu}. 
\end{align*}
 From the equality of $ \sigma B\left( \text{{\bf v}}_{\lambda} \right) $ 
 and $ B\left( \sigma \text{{\bf v}}_{\lambda} \right)$, 
 we compare the expansions and the coefficients in the $\text{{\bf v}}$-basis 
 for $ \sigma B\left( \text{{\bf v}}_{\lambda} \right) $ 
 and $ B\left( \sigma \text{{\bf v}}_{\lambda} \right)$, yielding 
\begin{align}
\begin{split}
 \sigma B\left( \text{{\bf v}}_{\lambda} \right)
 & = B\left( \sigma \text{{\bf v}}_{\lambda} \right) \\ 
 & = \sum_{\substack{ \mu \in \mathcal{P} \\ \ell(\mu) = k \\ \ell\left( \mu^{\text{T}} \right) \leq n }} 
 B^{\text{sort}\left( \sigma\left( \lambda_{1} \right), \sigma\left( \lambda_{2} \right), \ldots, \sigma\left( \lambda_{k} 
 \right) \right)}_{ \text{sort}\left( 
 \sigma\left( \mu_{1} \right), 
 \sigma\left( \mu_{2} \right), \ldots, 
 \sigma\left( \mu_{k} \right) \right) } 
 \text{{\bf v}}_{ \text{sort}\left( 
 \sigma\left( \mu_{1} \right), 
 \sigma\left( \mu_{2} \right), \ldots, 
 \sigma\left( \mu_{k} \right) \right) }, 
\end{split}\label{2024053012220202P2M22A}
\end{align}
 so that the desired result follows by comparing the coefficients in 
 \eqref{202405301159AM1A} and the coefficients in \eqref{2024053012220202P2M22A}. 
\end{proof}

 By analogy with Schur--Weyl duality for partition algebras, Lemma \ref{analogueJones} 
 gives us that a linear map $B\colon V^{\boxtimes k} \to V^{\boxtimes k}$ 
 is in $\text{End}_{S_{n}}\left( V^{\boxtimes k} \right)$ 
 if and only if the entries of the matrix corresponding to $B$
 are equal on orbits of $S_{n}$. 
 By analogy with how the dimension of the tensor power $V^{\otimes k}$
 is equal to $n^{k}$, we find that the dimension of 
 $V^{\boxtimes k}$ is equal to the number of integer partitions 
 satisfying the constraints shown in \eqref{basisboxpower}, 
 and this is easily seen to be equal to 
 $\binom{n+k-1}{n-1}$. 

\begin{definition}\label{definitionEunit}
 For integer partitions $\lambda$ and $\mu$ such that $\ell(\lambda) = \ell(\mu) = k$ and $\ell\left( \lambda^{\text{T}} \right) \leq n$ and $\ell\left( 
 \mu^{\text{T}} \right) \leq n$, let $E^{\lambda}_{\mu}$ denote the mapping in $\text{End}\left( V^{\boxtimes k} \right)$ that maps 
 $\text{{\bf v}}_{\lambda}$ to $\text{{\bf v}}_{\mu}$ and that maps $\text{{\bf v}}_{\nu}$ to $0$ for every integer partition $\nu$ 
 such that $\nu \neq \lambda$. 
\end{definition}

\begin{example}
 Letting $k = 2$ and $n = 4$, we let the $\text{{\bf v}}$-basis of $V^{\boxtimes k}$ be ordered according to the entries of the tuple 
 $( \text{{\bf v}}_{(1, 1)}, \text{{\bf v}}_{(2, 1)}$, $ \text{{\bf v}}_{(2, 2)}$, 
 $ \text{{\bf v}}_{(3, 1)}$, 
 $ \text{{\bf v}}_{(3, 2)}$, 
 $ \text{{\bf v}}_{(3, 3)}$, 
 $ \text{{\bf v}}_{(4, 1)}$, 
 $ \text{{\bf v}}_{(4, 2)}$, 
 $ \text{{\bf v}}_{(4, 3)}$, 
 $ \text{{\bf v}}_{(4, 4)} )$. 
 According to this ordering, the matrix corresponding to the linear morphism 
 $E^{(2, 1)}_{(1, 1)}$ is such that the $((2, 1), (1, 1))$-entry 
 is equal to $1$ and such that the remaining entries equal $0$. 
\end{example}

 The expansion in \eqref{writeBvlambda}
 may thus be rewritten so that 
 a given morphism $B \in \text{End}\left( V^{\boxtimes k} \right)$ may be expanded so that 
\begin{equation}\label{2027470753077407PM1A}
 B = \sum_{\substack{ \nu, \gamma \in \mathcal{P} \\ \ell(\nu) 
 = \ell(\gamma) = k \\ \ell\left( \nu^{\text{T}} \right) \leq n, \ell\left( \gamma^{\text{T}} \right) \leq n }} 
 B^{\nu}_{\gamma} E^{\nu}_{\gamma}. 
\end{equation}

\begin{definition}\label{definitionT}
 Letting $\lambda$ and $\mu$ satisfy the conditions in Definition \ref{definitionEunit}, we set 
\begin{equation}\label{20240530611PM1A}
 T^{\lambda}_{\mu} = \sum E^{\nu}_{\gamma}, 
\end{equation}
 where the sum in \eqref{20240530611PM1A} 
 is over all ordered pairs $(\nu, \gamma)$ 
 of integer partitions $\nu$ and $\gamma$
 such that $\ell(\nu) = \ell(\gamma) = k$ 
 and $\ell\left( \nu^{\text{T}} \right) \leq n$ 
 and $\ell\left( \gamma^{\text{T}} \right) \leq n$ 
 and there exists a permutation $\rho \in S_{n}$ 
 such that $\lambda = \text{sort}\big( \rho(\nu_{1}), \rho(\nu_{2}), \ldots, \rho(\nu_{k}) \big)$ 
 and $\mu = \text{sort}\big( \rho(\gamma_{1}), \rho(\gamma_{2}), \ldots, \rho(\gamma_{k}) \big)$. 
\end{definition}

\begin{example}\label{T2111E}
 For $k = 2$ and $n = 4$, the endomorphism $T^{(2, 1)}_{(1, 1)}$ satisfies 
\begin{align*}
 T^{(2, 1)}_{(1, 1)} 
 = & E^{(2, 1)}_{(1, 1)} + E^{(3, 1)}_{(1, 1)} + E^{(4, 1)}_{(1, 1)} + 
 E^{(3, 2)}_{(2, 2)} + E^{(4, 2)}_{(2, 2)} + E^{(2, 1)}_{(2, 2)} + 
 E^{(3, 1)}_{(3, 3)} + E^{(3, 2)}_{(3, 3)} + \\
 & E^{(4, 3)}_{(3, 3)} + 
 E^{(4, 1)}_{(4, 4)} + E^{(4, 2)}_{(4, 4)} + E^{(4, 3)}_{(4, 4)}. 
\end{align*}
 This is also equal to $T^{(3, 1)}_{(1, 1)}$ and $T^{(4, 1)}_{(1, 1)}$ and $T^{(2, 1)}_{(2, 2)}$, for example. 
\end{example}
 For an integer partition $p$ such $\ell\left( p^{\text{T}} \right) \leq n$, 
 we may rewrite $p$ as 
\begin{equation}\label{21000024707573707971727P7M71A}
 n^{m_{n}(p)} (n-1)^{m_{n-1}(p)} \cdots 1^{m_{1}(p)}. 
\end{equation}
 For $\sigma \in S_{n}$, we also may write
 $\sigma p$ in place of 
 $\text{sort}\big( \sigma(p_{1}), \sigma(p_{2}), \ldots, \sigma(p_{\ell(p)}) \big)$. 

\begin{lemma}\label{containmentlemma}
 Letting $\lambda$ and $\mu$ satisfy the conditions in Definition \ref{definitionEunit}, 
 we have that $T^{\lambda}_{\mu} \in \text{End}_{S_{n}}\left( V^{\boxtimes k} \right)$. 
\end{lemma}

\begin{proof}
 We rewrite \eqref{20240530611PM1A} according to \eqref{2027470753077407PM1A}, with 
\begin{equation}\label{20q274q707qq5q3qq07q7qq4q07qqPMq1A}
 T^{\lambda}_{\mu} = \sum_{\substack{ \nu, \gamma \in \mathcal{P} \\ \ell(\nu) 
 = \ell(\gamma) = k \\ \ell\left( \nu^{\text{T}} \right) \leq n, \ell\left( \gamma^{\text{T}} \right) \leq n }} 
 \left( T^{\lambda}_{\mu} \right)^{\nu}_{\gamma} E^{\nu}_{\gamma}. 
\end{equation}
 Let $p$ and $q$ be integer 
 partitions satisfying the conditions for the index set in \eqref{20q274q707qq5q3qq07q7qq4q07qqPMq1A}. 
 According to the index set for \eqref{20240530611PM1A}, 
 we have that $\left( T^{\lambda}_{\mu} \right)^{p}_{q} = 1$ if 
 there exists a permutation $\rho \in S_{n}$ 
 such that 
 $\lambda = \rho p$ 
 and $\mu = \rho q$, 
 with $\left( T^{\lambda}_{\mu} \right)^{p}_{q} = 0$ otherwise. 

 First, suppose that $\left( T^{\lambda}_{\mu} \right)^{\nu}_{\gamma} = 1$. 
 So, there exists a permutation $\rho \in S_{n}$ such that
 $\lambda = \rho \nu$ and $ \mu = \rho \gamma$. 
 Let $\sigma \in S_{n}$ be arbitrary. 
 By writing $\nu$ as in 
 \eqref{21000024707573707971727P7M71A}, we find that that 
 $\sigma \nu$ is obtained by sorting 
 $ (\sigma(n))^{m_{n}(\nu)} (\sigma(n-1))^{m_{n-1}(\nu)} \cdots (\sigma(1))^{m_{1}(\nu)}$, 
 so that a factor-by-factor application of $\sigma^{-1}$ gives us that 
 the application 
 of $\sigma^{-1}$ to 
 $\sigma \nu$ yields 
 $\nu$, and similarly for $\gamma$. 
 So, there exists a permutation $\rho \sigma^{-1} \in S_{n}$
 such that $\lambda = \rho \sigma^{-1}(\sigma \nu)$
 and such that 
 $\mu = \rho \sigma^{-1}(\sigma \gamma)$, 
 so that $ \left( T^{\lambda}_{\mu} 
 \right)^{\sigma \nu}_{ \sigma \gamma } = 1$, 
 according to the given characterization of the coefficients in \eqref{20q274q707qq5q3qq07q7qq4q07qqPMq1A}. 

 Now, suppose that $\left( T_{\mu}^{\lambda} \right)^{\nu}_{\gamma} = 0$. 
 So, it is not the case that there exists a permutation $\rho \in S_{n}$ 
 such that $\lambda = \rho \nu$ and $\mu = \rho \gamma$. 
 We again let $\sigma \in S_{n}$ be arbitrary. 
 If $\left( T_{\mu}^{\lambda} \right)^{\sigma \nu}_{\sigma \gamma} = 1$, 
 then there would exist a permutation $r \in S_{n}$ such that $\lambda = r \sigma \nu$
 and $\mu = r \sigma \gamma$, contradicting 
 an equivalent 
 formulation of $\left( T_{\mu}^{\lambda} \right)^{\nu}_{\gamma} = 0$. 
 So, $\left( T_{\mu}^{\lambda} \right)^{\sigma \nu}_{\sigma \gamma} = 0$, 
 so that Lemma \ref{analogueJones}
 gives us that $T^{\lambda}_{\mu} \in \text{End}_{S_{n}}\left( V^{\boxtimes k} \right)$. 
\end{proof}

\begin{definition}\label{pairsim}
 Let $(a, b)$ and $(c, d)$ 
 be pairs of partitions satisfying the index conditions in \eqref{2027470753077407PM1A}, 
 again for an endomorphism $B\colon V^{\boxtimes k} \to V^{\boxtimes k}$. 
 We let the relation $\sim$ be such that 
 $(a, b) \sim (c, d)$ is equivalent 
 to there existing a permutation $\sigma \in S_{n}$ 
 such that $a = \sigma c$ and $b = \sigma d$. 
\end{definition}

 The relation in Definition \ref{pairsim} 
 is an equivalence relation and is to be applied to obtain 
 an analogue of orbit bases for partition algebras. 

\begin{lemma}\label{2029405390191481P1M1A}
 The centralizer algebra $\text{End}_{S_{n}}\big( V^{\boxtimes k} \big)$ is spanned by the family of 
 expressions of the form $T^{\lambda}_{\mu}$, 
 for $\lambda$ and $\mu$ as 
 in Definition \ref{definitionT}. 
\end{lemma}

\begin{proof}
 Again for $B \in \text{End}\big( V^{\boxtimes k} \big)$, we rearrange the matrix unit expansion in 
 \eqref{2027470753077407PM1A} according to equivalence classes with respect to $\sim$. Let $\mathcal{C}$ denote an 
 equivalence class with respect to $\sim$. 
 According to our rearranged groupings of the terms in \eqref{2027470753077407PM1A}, 
 one of the groupings of terms we obtain is 
\begin{equation}\label{202747qq0q7q5q3qq0q7q7q4q0q7qPqMq1qA}
 \sum_{\substack{ a, b \in \mathcal{P} \\ \ell(a) 
 = \ell(b) = k \\ \ell\left( a^{\text{T}} \right) \leq n, \ell\left( b^{\text{T}} \right) \leq n \\ (a, b) \in \mathcal{C}
 }} B^{a}_{b} E^{a}_{b}. 
\end{equation}
 Now, we further assume that $B$ is in the centralizer algebra $\text{End}_{S_{n}}\big( V^{\boxtimes k} \big)$. 
 For pairs $(a, b)$ and $(c, d)$ appearing in the index set in \eqref{202747qq0q7q5q3qq0q7q7q4q0q7qPqMq1qA}, 
 Since $a = \sigma c$ and $b = \sigma d$, we have, by Lemma \ref{analogueJones}, 
 that $B^{a}_{b} = B^{c}_{d}$. 
 So, all of the coefficients in \eqref{202747qq0q7q5q3qq0q7q7q4q0q7qPqMq1qA} 
 are equal, for a given equivalence class $\mathcal{C}$. So, we may rewrite \eqref{202747qq0q7q5q3qq0q7q7q4q0q7qPqMq1qA} as 
\begin{equation}\label{coefficientoutisde}
 B^{\mathcal{C}} \sum_{\substack{ a, b \in \mathcal{P} \\ \ell(a) 
 = \ell(b) = k \\ \ell\left( a^{\text{T}} \right) \leq n, \ell\left( b^{\text{T}} \right) \leq n \\ (a, b) \in \mathcal{C}
 }} E^{a}_{b}, 
\end{equation} 
 writing $B^{\mathcal{C}} = B^{a}_{b}$ for $(a, b) \in \mathcal{C}$. For any $(e, f) \in \mathcal{C}$, we have, by Definition \ref{definitionT}, that 
 \eqref{coefficientoutisde} equals $ B^{\mathcal{C}} T^{e}_{f}$, which we rewrite, in a well defined way, as $ B^{\mathcal{C}} T^{\mathcal{C}}$. 
 From the foregoing considerations, we have that any morphism $B$ in $\text{End}_{S_{n}}\big( V^{\boxtimes k} \big)$ can be expanded as $ B = 
 \sum_{\mathcal{C}} B^{\mathcal{C}} T^{\mathcal{C}}$. 
\end{proof}

 As suggested in the proof of Lemma \ref{2029405390191481P1M1A}, 
 we extend the equivalence relation $\sim$ in Definition \ref{pairsim} 
 so as to write $T^{a}_{b} \sim T^{c}_{b}$ for $(a, b) \sim (c, d)$. 

\begin{theorem}\label{Cbasis}
 The family $\{ T^{\mathcal{C}} \}_{\mathcal{C}}$ indexed by equivalence classes $\mathcal{C}$ 
 with respect to $\sim$ on the set of pairs $(\nu, \gamma)$ of integer partitions 
 such that $\ell(\nu) = \ell(\gamma) = k$ and $\ell\left( \nu^{\text{T}} \right) \leq n$ 
 and $\ell\left( \gamma^{\text{T}} \right) \leq n$ 
 is a basis of $\text{End}_{S_{n}}\big( V^{\boxtimes k} \big)$. 
\end{theorem}

\begin{proof}
 By Lemma \ref{containmentlemma}, 
 we have that $\{ T^{\mathcal{C}} \}_{\mathcal{C}} \subseteq \text{End}_{S_{n}}\big( V^{\boxtimes k} \big)$. 
 By Lemma \ref{2029405390191481P1M1A}, we have that 
 $\text{span} \{ T^{\mathcal{C}} \}_{\mathcal{C}} = \text{End}_{S_{n}}\big( V^{\boxtimes k} \big)$. 
 By Definition \ref{definitionT}, each expression of the form $T^{\mathcal{C}}$ 
 is a linear combination of matrix units of $\text{End}\big( V^{\boxtimes} \big)$, 
 yielding a linearly independent spanning subset. 
\end{proof}

\begin{example}\label{dimension9example}
 As in Example \ref{T2111E}, we set $k = 2$ and $n = 4$. 
 As illustrated in Example \ref{T2111E}, we have that 
 $T^{(2, 1)}_{(1, 1)} = T^{(3, 1)}_{(1, 1)} = T^{(4, 1)}_{(1, 1)} = T^{(2, 1)}_{(2, 2)}$. 
 According to the notation in Theorem \ref{Cbasis}, 
 each of the equivalent expressions among 
 $T^{(2, 1)}_{(1, 1)}$, $ T^{(3, 1)}_{(1, 1)}$, $ T^{(4, 1)}_{(1, 1)}$ and $ T^{(2, 1)}_{(2, 2)}$
 is equal to an expression of the form $T^{\mathcal{C}}$, for an equivalence class
 of the form specified in Theorem \ref{Cbasis}. 
 It can be shown that a basis of $\text{End}_{S_{n}}\big( V^{\boxtimes k} \big)$ is 
\begin{equation}\label{basisk2n4}
 \left\{ T^{(2, 1)}_{(4, 3)}, T^{(1, 1)}_{(3, 2)}, T^{(2, 1)}_{(3, 1)}, 
 T^{(3, 2)}_{(1, 1)}, T^{(1, 1)}_{(2, 2)}, T^{(2, 1)}_{(2, 1)}, T^{(1, 1)}_{(2, 1)}, 
 T^{(2, 1)}_{(1, 1)}, T^{(1, 1)}_{(1, 1)} \right\}. 
\end{equation}
 With regard to the basis element $ T^{(2, 1)}_{(4, 3)}$ in \eqref{basisk2n4}, observe the significance of $n$ being greater than or equal to $2k$. Also 
 observe that the dimension of $\text{End}_{S_{n}}\big( V^{\boxtimes k} \big)$ is strictly less than that for $\text{End}_{S_{n}}\big( V^{\otimes k} \big)$, 
 for the $k = 2$ and $n = 4$ case under consideration. 
\end{example}

 With regard to our notation for set-partitions, we typically let
 set-partitions be of $\{ 1, 2, \ldots, k \} \cup \{ 1', 2', \ldots, k' \}$, 
 where $\{ 1, 2, \ldots, k \}$ and $\{ 1', 2', \ldots, k' \}$
 are disjoint, with 
\begin{equation}\label{primeordering}
 1 < 2 < \cdots < k < 1' < 2' < \cdots < k'. 
\end{equation}
 We also adopt the convention whereby 
 the elements in a given set-partition of $\{ 1, 2, \ldots, k \} \cup \{ 1', 2', \ldots, k' \}$ 
 are ordered reverse-lexicographically, with respect to \eqref{primeordering}. 

\begin{definition}\label{definitionlozenge}
 Let $\mathcal{S}$ and $\mathcal{T}$ be set-partitions of $\{ 1, 2, \ldots, k, 1', 2', \ldots, k' \}$. Define the relation $\wasylozenge$ so that $ 
 \mathcal{S} \wasylozenge \mathcal{T}$ is equivalent to there being permutations $\alpha \in S_{k}$ and $\beta \in S_k$
 such that: By replacing the unique occurrence of $i$ in $\mathcal{S}$ 
 with $\alpha(i)$ for $i \in \{ 1, 2, \ldots, k \}$ and by replacing the unique occurrence of 
 $j'$ in $\mathcal{S}$ with $(\beta(j))'$ for $j \in \{ 1, 2, \ldots, k \}$, 
 the resultant set-partition equals $\mathcal{T}$. 
\end{definition}

\begin{example}
 According to Definition \ref{definitionlozenge}, 
 we have that $\{ \{ 1 \}$, $ \{ 2$, $ 1'$, $ 2' \} \}$ $ \wasylozenge $ $ \{ \{ 2 \}$, $ \{ 1$, $ 1'$, $ 2' \} \}$. 
\end{example}

 The relation $\wasylozenge $ in Definition \ref{definitionlozenge} is an equivalence relation, and this provides a key to our combinatorial characterization of 
 the basis elements of $\text{End}_{S_{n}}\big( V^{\boxtimes k} \big)$. To begin with, we adopt the usual graphical notation described below for 
 partition diagrams. Mimicking notation from Beyene et al.\ \cite{BeyeneBackelinMantaciFufa2023}, we let $\text{SP}(k, k')$ denote the set of all set-partitions of $\{ 1, 
 2, \ldots, k, 1', 2', \ldots, k' \}$. 

\begin{definition}
 Let $\mathcal{S}$ be a set-partition in $\text{SP}(k, k')$. 
 We consider two graphs on $\{ 1, 2, \ldots, k, 1', 2', \ldots, k' \}$ 
 be considered as equivalent if the connected components of these graphs are the same. 
 For a graph $G$ such that the components of $G$ 
 agree with the elements of $\mathcal{S}$, 
 we let the \emph{partition diagram} of $\mathcal{S}$ 
 refer to the equivalence class of $G$. 
 This equivalence class may be denoted with any graph in it, 
 with vertices labeled with $1 < 2 < \cdots < k$ 
 arranged into a top row and with 
 vertices labeled with $1' < 2' < \cdots k'$ 
 arranged into a bottom row. 
\end{definition}

 We may denote, as below, a given set-partition with any graph in the corresponding partition diagram. 

\begin{example}
 The partition diagram for $\{ \{ 1, 2, 3, 5' \}, \{ 4, 5 \}, \{ 1', 2' \}, \{ 3', 4' \} \}$ 
 may be denoted with the following graph: 
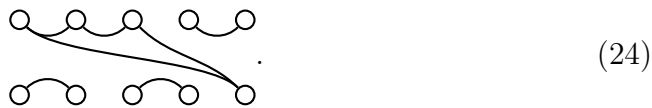
\begin{equation}\label{20240531937PM1A}
 \begin{tikzpicture}[scale = 0.5,thick, baseline={(0,-1ex/2)}] 
\tikzstyle{vertex} = [shape = circle, minimum size = 7pt, inner sep = 1pt] 
\node[vertex] (G--5) at (6.0, -1) [shape = circle, draw] {}; 
\node[vertex] (G-1) at (0.0, 1) [shape = circle, draw] {}; 
\node[vertex] (G-2) at (1.5, 1) [shape = circle, draw] {}; 
\node[vertex] (G-3) at (3.0, 1) [shape = circle, draw] {}; 
\node[vertex] (G--4) at (4.5, -1) [shape = circle, draw] {}; 
\node[vertex] (G--3) at (3.0, -1) [shape = circle, draw] {}; 
\node[vertex] (G--2) at (1.5, -1) [shape = circle, draw] {}; 
\node[vertex] (G--1) at (0.0, -1) [shape = circle, draw] {}; 
\node[vertex] (G-4) at (4.5, 1) [shape = circle, draw] {}; 
\node[vertex] (G-5) at (6.0, 1) [shape = circle, draw] {}; 
\draw[] (G-1) .. controls +(0.5, -0.5) and +(-0.5, -0.5) .. (G-2); 
\draw[] (G-2) .. controls +(0.5, -0.5) and +(-0.5, -0.5) .. (G-3); 
\draw[] (G-3) .. controls +(1, -1) and +(-1, 1) .. (G--5); 
\draw[] (G--5) .. controls +(-1, 1) and +(1, -1) .. (G-1); 
\draw[] (G--4) .. controls +(-0.5, 0.5) and +(0.5, 0.5) .. (G--3); 
\draw[] (G--2) .. controls +(-0.5, 0.5) and +(0.5, 0.5) .. (G--1); 
\draw[] (G-4) .. controls +(0.5, -0.5) and +(-0.5, -0.5) .. (G-5); 
\end{tikzpicture}. 
\end{equation} 
 According to the equivalence relation in Definition \ref{definitionlozenge}, 
 we have, via the application of the permutation 
 $\left(\begin{smallmatrix} 
1 & 2 & 3 & 4 & 5 \\ 
 2& 3 & 4 & 5 & 1
\end{smallmatrix}\right)$ 
 to the unprimed integers in the above set-partition and by applying 
 $\left(\begin{smallmatrix} 
1 & 2 & 3 & 4 & 5 \\ 
 1& 3 & 2 & 4 & 5
\end{smallmatrix}\right)$ 
 to the primed integers, we obtain that 
 $\{ \{ 1$, $2$, $ 3$, $ 5' \}$, $ \{ 4$, $ 5 \}$, $ \{ 1'$, $ 2' \}$, $ \{ 3'$, $ 4' \} \} $
 $ \wasylozenge $ $ \{ \{ 2$, $ 3$, $ 4$, $ 5' \}$, $ \{ 5$, $ 1 \}$, $ \{ 1'$, $ 3' \}$, $ \{ 2'$, $ 4' \} \}$. 
 This latter set-partition may be denoted as 
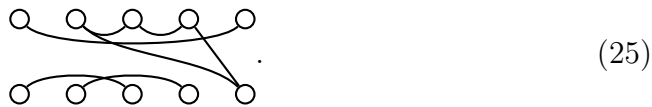
\begin{equation}\label{202405319q42qqtP2M2A}
\begin{tikzpicture}[scale = 0.5,thick, baseline={(0,-1ex/2)}] 
\tikzstyle{vertex} = [shape = circle, minimum size = 7pt, inner sep = 1pt] 
\node[vertex] (G--5) at (6.0, -1) [shape = circle, draw] {}; 
\node[vertex] (G-2) at (1.5, 1) [shape = circle, draw] {}; 
\node[vertex] (G-3) at (3.0, 1) [shape = circle, draw] {}; 
\node[vertex] (G-4) at (4.5, 1) [shape = circle, draw] {}; 
\node[vertex] (G--4) at (4.5, -1) [shape = circle, draw] {}; 
\node[vertex] (G--2) at (1.5, -1) [shape = circle, draw] {}; 
\node[vertex] (G--3) at (3.0, -1) [shape = circle, draw] {}; 
\node[vertex] (G--1) at (0.0, -1) [shape = circle, draw] {}; 
\node[vertex] (G-1) at (0.0, 1) [shape = circle, draw] {}; 
\node[vertex] (G-5) at (6.0, 1) [shape = circle, draw] {}; 
\draw[] (G-2) .. controls +(0.5, -0.5) and +(-0.5, -0.5) .. (G-3); 
\draw[] (G-3) .. controls +(0.5, -0.5) and +(-0.5, -0.5) .. (G-4); 
\draw[] (G-4) .. controls +(0.75, -1) and +(-0.75, 1) .. (G--5); 
\draw[] (G--5) .. controls +(-1, 1) and +(1, -1) .. (G-2); 
\draw[] (G--4) .. controls +(-0.6, 0.6) and +(0.6, 0.6) .. (G--2); 
\draw[] (G--3) .. controls +(-0.6, 0.6) and +(0.6, 0.6) .. (G--1); 
\draw[] (G-1) .. controls +(0.8, -0.8) and +(-0.8, -0.8) .. (G-5); 
\end{tikzpicture}.
\end{equation}
\end{example}

\begin{definition}\label{definitionMS}
 Let $\mathcal{S}$ be a set-partition in $\text{SP}(k, k')$. 
 Let $G$ be any graph in the equivalence class given by the partition diagram for $\mathcal{S}$. 
 We list the distinct components of $G$ according to 
 the order in which these components appear 
 among the ordered vertices $1 < 2 < \cdots < k < 1' < 2' < \cdots < k'$, i.e., 
 by first listing the unique component containing $1$, and by then listing a component, if it exists, 
 containing the smallest vertex outside of the first component, etc. 
 We then relabel the vertices of $G$, 
 by labeling the vertices in the $i^{\text{th}}$ component with $i$. 
 We then define the integer matrix $M(\mathcal{S})$ corresponding to $\mathcal{S}$ 
 as the $2 \times k$ matrix obtained from 
 $\left(\begin{smallmatrix} 
1 & 2 & \cdots & k \\ 
 1' & 2' & \cdots & k' 
\end{smallmatrix}\right)$ 
 by replacing each expression $e$ appearing among the entries with 
 the new label for the vertex originally labelled with $e$ in $G$. 
\end{definition}

\begin{example}
 With regard to the partition diagrams shown in 
 \eqref{20240531937PM1A} and \eqref{202405319q42qqtP2M2A}, we find that 
\begin{equation}\label{Mgraphfirst}
 M\left( \begin{tikzpicture}[scale = 0.5,thick, baseline={(0,-1ex/2)}] 
\tikzstyle{vertex} = [shape = circle, minimum size = 7pt, inner sep = 1pt] 
\node[vertex] (G--5) at (6.0, -1) [shape = circle, draw] {}; 
\node[vertex] (G-1) at (0.0, 1) [shape = circle, draw] {}; 
\node[vertex] (G-2) at (1.5, 1) [shape = circle, draw] {}; 
\node[vertex] (G-3) at (3.0, 1) [shape = circle, draw] {}; 
\node[vertex] (G--4) at (4.5, -1) [shape = circle, draw] {}; 
\node[vertex] (G--3) at (3.0, -1) [shape = circle, draw] {}; 
\node[vertex] (G--2) at (1.5, -1) [shape = circle, draw] {}; 
\node[vertex] (G--1) at (0.0, -1) [shape = circle, draw] {}; 
\node[vertex] (G-4) at (4.5, 1) [shape = circle, draw] {}; 
\node[vertex] (G-5) at (6.0, 1) [shape = circle, draw] {}; 
\draw[] (G-1) .. controls +(0.5, -0.5) and +(-0.5, -0.5) .. (G-2); 
\draw[] (G-2) .. controls +(0.5, -0.5) and +(-0.5, -0.5) .. (G-3); 
\draw[] (G-3) .. controls +(1, -1) and +(-1, 1) .. (G--5); 
\draw[] (G--5) .. controls +(-1, 1) and +(1, -1) .. (G-1); 
\draw[] (G--4) .. controls +(-0.5, 0.5) and +(0.5, 0.5) .. (G--3); 
\draw[] (G--2) .. controls +(-0.5, 0.5) and +(0.5, 0.5) .. (G--1); 
\draw[] (G-4) .. controls +(0.5, -0.5) and +(-0.5, -0.5) .. (G-5); 
\end{tikzpicture} \right) = \left( \begin{matrix}
 1 & 1 & 1 & 2 & 2 \\
 3 & 3 & 4 & 4 & 1 
 \end{matrix} \right)
\end{equation}
 and that 
\begin{equation}\label{Mgraphsecond}
 M\left( \begin{tikzpicture}[scale = 0.5,thick, baseline={(0,-1ex/2)}] 
\tikzstyle{vertex} = [shape = circle, minimum size = 7pt, inner sep = 1pt] 
\node[vertex] (G--5) at (6.0, -1) [shape = circle, draw] {}; 
\node[vertex] (G-2) at (1.5, 1) [shape = circle, draw] {}; 
\node[vertex] (G-3) at (3.0, 1) [shape = circle, draw] {}; 
\node[vertex] (G-4) at (4.5, 1) [shape = circle, draw] {}; 
\node[vertex] (G--4) at (4.5, -1) [shape = circle, draw] {}; 
\node[vertex] (G--2) at (1.5, -1) [shape = circle, draw] {}; 
\node[vertex] (G--3) at (3.0, -1) [shape = circle, draw] {}; 
\node[vertex] (G--1) at (0.0, -1) [shape = circle, draw] {}; 
\node[vertex] (G-1) at (0.0, 1) [shape = circle, draw] {}; 
\node[vertex] (G-5) at (6.0, 1) [shape = circle, draw] {}; 
\draw[] (G-2) .. controls +(0.5, -0.5) and +(-0.5, -0.5) .. (G-3); 
\draw[] (G-3) .. controls +(0.5, -0.5) and +(-0.5, -0.5) .. (G-4); 
\draw[] (G-4) .. controls +(0.75, -1) and +(-0.75, 1) .. (G--5); 
\draw[] (G--5) .. controls +(-1, 1) and +(1, -1) .. (G-2); 
\draw[] (G--4) .. controls +(-0.6, 0.6) and +(0.6, 0.6) .. (G--2); 
\draw[] (G--3) .. controls +(-0.6, 0.6) and +(0.6, 0.6) .. (G--1); 
\draw[] (G-1) .. controls +(0.8, -0.8) and +(-0.8, -0.8) .. (G-5); 
\end{tikzpicture} \right) = \left( \begin{matrix}
 1 & 2 & 2 & 2 & 1 \\
 3 & 4 & 3 & 4 & 2 
 \end{matrix} \right). 
\end{equation} 
\end{example}

\begin{definition}
 We again let $\mathcal{S} \in \text{SP}(k, k')$. 
 We define $\lambda(\mathcal{S})$ and $\mu(\mathcal{S})$, respectively, 
 as the integer partitions 
 obtained by sorting the entries of the upper and lower rows of $M(\mathcal{S})$. 
\end{definition}

\begin{example}
 With regard to the matrix evaluations in \eqref{Mgraphfirst} 
 and \eqref{Mgraphsecond}, 
 we set $\mathcal{S}$ $=$ $ \{ \{ 1$, $2$, $3$, $ 5' \}$, $ \{ 4$, $ 5 \}$, $ \{ 1'$, $ 2' \}$, $ \{ 3'$, $ 4' \} \}$ 
 and $\mathcal{T}$ $ = $ $ \{ \{ 1$, $ 5 \}$, $ \{ 2$, $ 3$, $ 4$, $ 5' \}$, $ \{ 1'$, $ 3' \}$, $ \{ 2'$, $ 4' \} \}$, 
 with $\mathcal{S} \wasylozenge \mathcal{T}$. This yields 
 $ \lambda(\mathcal{S}) = (2, 2, 1, 1, 1)$, 
 $ \mu(\mathcal{S}) = (4, 4, 3, 3, 1)$, 
 $ \lambda(\mathcal{T}) = (2, 2, 2, 1, 1)$, 
 and $ \mu(\mathcal{T}) = (4, 4, 3, 3, 2)$. 
 Setting $n = 10$, with $2k \leq n$, we find that 
 $T^{\lambda(\mathcal{S})}_{\mu(\mathcal{S})} \sim T^{\lambda(\mathcal{T})}_{\mu(\mathcal{T})}$. 
\end{example}

\begin{theorem}\label{theoremfirstdimension}
 If $2k \leq n$, then the dimension of $\text{End}_{S_{n}}\big( V^{\boxtimes k} \big)$ 
 is equal to the number of equivalence classes with respect to $\wasylozenge$ 
 of the set of set-partitions of $\{ 1, 2, \ldots, k, 1', 2', \ldots, k' \}$. 
\end{theorem}

\begin{proof}
 Our strategy is to construct a bijection $\varphi\colon \text{SP}(k, k')/ \wasylozenge \to
 \{ T^{\mathcal{C}} \}_{\mathcal{C}}$, 
 for the family $ \{ T^{\mathcal{C}} \}_{\mathcal{C}}$ 
 satisfying the conditions specified in Theorem \ref{Cbasis}. 

 Letting $\mathcal{S}$ denote an element in $\text{SP}(k, k')$, 
 we set 
\begin{equation}\label{202406011011AM1A}
 \varphi\left( [ \mathcal{S} ]_{\wasylozenge} \right) = \left[ T^{\lambda(\mathcal{S})}_{\mu(\mathcal{S})} \right]_{\sim}. 
\end{equation}
 Let $\mathcal{S}$ and $\mathcal{T}$ be elements in $\text{SP}(k, k')$ such that $[\mathcal{S}]_{\wasylozenge} = [ \mathcal{T}]_{\wasylozenge}$. 
 Since $\mathcal{S} \wasylozenge \mathcal{T}$, if we take graphs $G_{\mathcal{S}}$ and $G_{\mathcal{T}}$ in the equivalence classes given, 
 respectively, by the partition diagrams for $\mathcal{S}$ and $\mathcal{T}$, 
 the graph $G_{\mathcal{T}}$ can be obtained from 
 $G_{\mathcal{S}}$ by rearranging the ordering of the the top (resp.\ bottom) vertices of $G_{\mathcal{S}}$ 
 while preserving adjacent vertices being adjacent and preserving
 non-adjacent vertices being non-adjacent. 
 So, since $2k \leq n$, there exists a permutation $\sigma \in S_{n}$ such that 
 the entries in the top (resp.\ bottom) row of $M(\mathcal{T})$ 
 are given by a possible rearrangement of the ordering of the entries 
 of the tuple given by applying $\sigma$ to the top (resp.\ bottom) 
 row of $M(\mathcal{S})$ . 
 In other words, we have that 
 the same permutation $\sigma \in S_{n}$ 
 satisfies $\lambda(\mathcal{S}) = \sigma \lambda(\mathcal{T})$ 
 and $\mu(\mathcal{S}) = \sigma \mu(\mathcal{T})$. 
 Since $\mathcal{S}, \mathcal{T} \in \text{SP}(k, k')$, 
 we have that $\ell(\lambda(\mathcal{S})) = \ell(\lambda(\mathcal{T})) = \ell(\mu(\mathcal{S}))
 = \ell(\mu(\mathcal{T})) = k$. Since $2k \leq n$, 
 we have that the entries of $\lambda(\mathcal{S})$, $\lambda(\mathcal{T})$, $\mu(\mathcal{S})$, 
 and $\mu(\mathcal{T})$ 
 are all less than or equal to $n$. 
 So, from Definition \ref{pairsim}, 
 we have that 
 $(\lambda(\mathcal{S}), \mu(\mathcal{S})) \sim (\lambda(\mathcal{T}), \mu(\mathcal{T}))$. 
 This gives us that $T^{\lambda(\mathcal{S})}_{\mu(\mathcal{S})} \sim T^{\lambda(\mathcal{T})}_{\mu(\mathcal{T})}$. 
 This gives us that $\varphi$ is well defined, according to the definition in \eqref{202406011011AM1A}. 

 Let $\mathcal{S}$ and $\mathcal{T}$ be members of $\text{SP}(k, k')$. 
 Suppose that $\varphi( [ \mathcal{S} ]_{\wasylozenge} ) 
 = \varphi( [ \mathcal{T} ]_{\wasylozenge} )$. 
 So, we have that 
 $\big[ T^{\lambda(\mathcal{S})}_{\mu(\mathcal{S})} \big]_{\sim} = 
 \big[ T^{\lambda(\mathcal{T})}_{\mu(\mathcal{T})} \big]_{\sim}$. 
 This gives us that 
 $ T^{\lambda(\mathcal{S})}_{\mu(\mathcal{S})} \sim T^{\lambda(\mathcal{T})}_{\mu(\mathcal{T})}$. 
 In turn, we obtain 
 $(\lambda(\mathcal{S}), \mu(\mathcal{S})) \sim (\lambda(\mathcal{T}), \mu(\mathcal{T}))$. 
 So, if we compare the respective entries of the matrices 
 $M(\mathcal{S})$ and $M(\mathcal{T})$, 
 we find that there is a permutation $\sigma \in S_{n}$ such that 
 $M(\mathcal{S})$ is equivalent to $ \sigma M(\mathcal{T})$, 
 up to rearrangements of the entries in the top row 
 and possibly different rearrangements of the entries in the bottom row. 
 So, by identifying $\mathcal{S}$ (resp.\ $\mathcal{T}$) 
 with a graph $G_{\mathcal{S}}$ (resp.\ $G_{\mathcal{T}}$) in the partition diagram for $\mathcal{S}$
 (resp.\ $\mathcal{T}$), 
 by labeling the vertices in $G_{\mathcal{S}}$ (resp.\ $G_{\mathcal{T}}$)
 according to the entries of $M(G_{\mathcal{S}})$ (resp.\ $M(G_{\mathcal{T}})$), 
 we find that $G_{\mathcal{S}}$ is obtained by relabelling $G_{\mathcal{T}}$
 (while maintaining the property that vertices in the same component have the same labels and while maintaining
 that different components have different labels) 
 and by then rearranging the ordering of the vertices in the top row
 and by then rearranging the ordering of the vertices in the bottow row. 
 By identifying these vertex order rearrangements with the permutations $\alpha \in S_{k}$ 
 and $\beta \in S_{k}$ in Definition \ref{definitionlozenge}, 
 we find that $\mathcal{S} \wasylozenge \mathcal{T}$ 
 since the components of $\mathcal{S}$ and $\mathcal{T}$ 
 are preserved according to the given relabelling. This gives us that $\varphi$ is injective. 

 Let $\nu$ and $\gamma$ be integer partitions such that $\ell(\nu) = \ell(\gamma) = k$ and $\ell\left( \nu^{\text{T}} \right) \leq n$ and $\ell\left( 
 \gamma^{\text{T}} \right) \leq n$. So, the equivalence class $\big[ T^{\nu}_{\gamma} \big]_{\sim}$ is an arbitrary element in the specified codomain of 
 $\varphi$. We construct a set-partition $\mathcal{S} \in \text{SP}(k, k')$ in the following manner. For a positive integer $i \leq n$, we let $ 
 \text{upper}_{i}$ and $\text{lower}_{i}$ respectively denote the number of copies of $i$ in $\nu$ and $\gamma$. We then let $\{ 1, 2, \ldots, 
 \text{upper}_{1} \} \cup \{ 1', 2', \ldots, (\text{lower}_{1})' \}$ be a component in $\mathcal{S}$, and we then let $\{ \text{upper}_{1}, 
 \text{upper}_{1} + 1, \ldots, \text{upper}_{1} + \text{upper}_{2} \} \cup \{ (\text{lower}_{1})', (\text{lower}_{1} + 1)', \ldots, (\text{lower}_{1} + 
 \text{lower}_{2})' \}$ be a component in $\mathcal{S}$, and so forth. Following the procedure in Definition \ref{definitionMS}, 
 we find that $M(\mathcal{S})$ is such that its upper row, up to a rearrangement of entries, 
 is $\nu$ and that its lower row, up to a possibly different rearrangement, is $\gamma$. 
 We thus find that $\varphi( [\mathcal{S}]_{\wasylozenge} ) = \big[ T^{\nu}_{\gamma} \big]_{\sim}$. 
 We thus have that $\varphi$ is surjective. 
\end{proof}

 Our main result is highlighted as the below Theorem.
 To begin with, we introduce some notation required for this Theorem. 
 For a set $S$ of natural numbers, we let $S'$ denote the set obtained from $S$ 
 by ``priming'' the elements in $S$. 
 A \emph{perfect matching} of a finite set $S$ is a set-partition of $S$ such that every 
 element in this set-partition is of cardinality $2$. 
 For two finite and disjoint sets $S$ and $T$, we 
 let $ \text{match}(S, T) $ denote the set of all set-partitions $P = P(S, T)$ of $S \cup T$ 
 such that each element in $P$ is either a single set or a set of cardinality $2$ 
 and such that any $2$-element est $Q$ in $P$ 
 contains exactly one element in $S$ and exactly one element in $T$. 
 A set-partition $P = (S, T)$ in $\text{match}(S, T)$ may be identified with 
 a bipartite graph $ G_{P(S, T)} = G(S, T, E_{P})$, where the edge set $E = E_{P} = E_{P(S, T)}$ consists of the $2$-sets 
 in $P$. For an edge $e \in E_{P(S, T)}$, 
 we write $e = \{ e_{S}, e_{T} \}$, 
 letting $e_{S}$ and $e_{T}$ respectively 
 denote the unique elements in $e \cap S$ 
 and $e \cap T$. 

\begin{theorem}\label{dimensiontheorem}
 The dimension of $\text{End}_{S_{n}}\big( V^{\boxtimes k} \big)$ is equal to 
 \begin{equation}\label{mainformula}
 \sum_{\substack{ (\lambda, \mu) \in \mathcal{P}^{2} \\ \lambda, \mu \vdash k \\ 
 P \in \text{match}(\text{distinct}(\lambda), \text{distinct}(\mu)') }} 
 \prod_{e \in E_{P}} \min\{ m_{e_{\text{distinct}(\lambda)}}(\lambda), m_{e_{\text{distinct}(\mu)}}(\mu) \} 
\end{equation} 
 for each positive integer $k$, provided that $2 k \leq n$. 
\end{theorem}

\begin{proof}
 We take a set-partition $P = P(\text{distinct}(\lambda), \text{distinct}(\mu)')$ in $$ \text{match}(\text{distinct}(\lambda), \text{distinct}(\mu)') $$ and, as 
 above, we identify $P$ with the bipartite graph $$G_{P} = G(\text{distinct}(\lambda), \text{distinct}(\mu)', E_{P}),$$ by identifying $\lambda \vdash k$ 
 with a partitioning of the vertices $1 < 2 < \cdots < k$ into connected components arranged in non-increasing order according to size, and similarly for 
 $ \mu \vdash k$, with respect to vertices $1' < 2' < \cdots < k'$ arranged into a bottom row. In this regard, for any edge $e \in E_{P}$, 
 this can be thought of, informally, as indicating that at least one of the upper components 
 of size equal to 
 equal to $e_{\text{distinct}(\lambda)}$ is to be joined with at least one of the bottom 
 components of size equal to $e_{\text{distinct}(\mu)}$. 

 By the invariance property given by equivalence according to $\wasylozenge$, partition diagrams are equivalent up to rearrangements of the upper vertices 
 and possibly different rearrangements of the lower vertices, without changing the connected components. So, if there is a component consisting of $u$ 
 upper vertices and $l$ lower vertices, and a different component that also consists of $u$ upper vertices and $l$ lower vertices, then these components 
 are equivalent in terms of rearrangements of vertices of the upper row and possibly different rearrangements of vertices in the lower row. So, for $e \in 
 E_{P}$ as above, there are 
\begin{equation}\label{minmultiplicity}
 \min\{ m_{e_{\text{distinct}(\lambda)}}(\lambda), m_{e_{\text{distinct}(\mu)}}(\mu) \} 
\end{equation}
 possibilities to consider: We could 
 take one out of the $ m_{e_{\text{distinct}(\lambda)}}(\lambda)$ copies of upper components of size $e_{\text{distinct}(\lambda)}$ and join 
 it with one 
 out of the $ m_{e_{\text{distinct}(\mu)}}(\mu)$ copies of lower components of size $e_{\text{distinct}(\mu)}$, but, by taking, if possible, 
 another one out of the upper components of the specified form and another one of the lower components of the specified form and joining these upper 
 and lower components, this is indistinguishable from the previously formed component, up to rearrangements of upper vertices and possibly different 
 rearrangements of lower vertices. 

 The foregoing considerations give rise to a bijection, via Theorem \ref{theoremfirstdimension}. For a set-partition $\mathcal{S}$ in $\text{SP}(k, k')$, 
 we fix the element from $[\mathcal{S}]_{\wasylozenge}$ obtained by replacing each component of a graph in the partition diagram of $\mathcal{S}$ 
 with a complete graph, and then forming a partition of the upper (resp.\ lower) rows by removing any edges in between the rows, and by then sorting the 
 resultant partitions according to the non-increasing size of the parts, and by then permuting the upper and lower rows of a graph in the partition 
 diagram for $\mathcal{S}$ according to the same rearrangements. This gives us a bijection between $\text{SP}(k, k') / \wasylozenge$ and the 
 set of equivalence class representatives obtained in the manner specified: The elements in this latter set are uniquely identified by the integer 
 partitions $\lambda$ and $\mu$ of $k$ obtained by sorting the parts of the upper and lower rows after removing any propagating edges, together 
 with any edges incident with upper and lower parts, up to equivalence whereby two edges incident with distinct upper parts of the same size $s_{1}$ 
 and incident with distinct lower parts of a given size $s_{2}$ 
 give rise to equivalent diagrams, according to $\wasylozenge$. 
 So, by taking all possible pairs $(\lambda, \mu)$ of integer partitions such that $\lambda \vdash k$ and $\mu \vdash k$, 
 and by taking all possible bipartite graphs
 of the form $\text{match}(\text{distinct}(\lambda), \text{distinct}(\mu)')$, 
 by then summing over products of the form indicated in \eqref{mainformula}, 
 the different possibilities given by the minimal value in \eqref{minmultiplicity} 
 give us the number of equivalence classes with respect to $\wasylozenge$. 
\end{proof}

\begin{example}
 Let $k = 3$ and let $2k \leq n$. According to the bijective proof of Theorem \ref{theoremfirstdimension}, 
 we may identify a given basis element in $\text{End}_{S_{n}}\big( V^{\boxtimes k} \big)$ 
 with an equivalence class with respect to $\wasylozenge$, 
 and hence our below notation for basis elements, arranged below according to 
 pairs of the form $(\lambda, \mu)$ involved in the index set in \eqref{mainformula}

 \ 

 \ 

\noindent $\lambda = (3)$, $\mu = (3)$: \ 

\noindent $ \left[ \begin{tikzpicture}[scale = 0.5,thick, baseline={(0,-1ex/2)}] 
\tikzstyle{vertex} = [shape = circle, minimum size = 7pt, inner sep = 1pt] 
\node[vertex] (G--3) at (3.0, -1) [shape = circle, draw] {}; 
\node[vertex] (G--2) at (1.5, -1) [shape = circle, draw] {}; 
\node[vertex] (G--1) at (0.0, -1) [shape = circle, draw] {}; 
\node[vertex] (G-1) at (0.0, 1) [shape = circle, draw] {}; 
\node[vertex] (G-2) at (1.5, 1) [shape = circle, draw] {}; 
\node[vertex] (G-3) at (3.0, 1) [shape = circle, draw] {}; 
\draw[] (G-1) .. controls +(0.5, -0.5) and +(-0.5, -0.5) .. (G-2); 
\draw[] (G-2) .. controls +(0.5, -0.5) and +(-0.5, -0.5) .. (G-3); 
\draw[] (G-3) .. controls +(0, -1) and +(0, 1) .. (G--3); 
\draw[] (G--3) .. controls +(-0.5, 0.5) and +(0.5, 0.5) .. (G--2); 
\draw[] (G--2) .. controls +(-0.5, 0.5) and +(0.5, 0.5) .. (G--1); 
\draw[] (G--1) .. controls +(0, 1) and +(0, -1) .. (G-1); 
 \end{tikzpicture} \right]_{\wasylozenge} $
  \ \ \  $ \left[ \begin{tikzpicture}[scale = 0.5,thick, baseline={(0,-1ex/2)}] 
\tikzstyle{vertex} = [shape = circle, minimum size = 7pt, inner sep = 1pt] 
\node[vertex] (G--3) at (3.0, -1) [shape = circle, draw] {}; 
\node[vertex] (G--2) at (1.5, -1) [shape = circle, draw] {}; 
\node[vertex] (G--1) at (0.0, -1) [shape = circle, draw] {}; 
\node[vertex] (G-1) at (0.0, 1) [shape = circle, draw] {}; 
\node[vertex] (G-2) at (1.5, 1) [shape = circle, draw] {}; 
\node[vertex] (G-3) at (3.0, 1) [shape = circle, draw] {}; 
\draw[] (G--3) .. controls +(-0.5, 0.5) and +(0.5, 0.5) .. (G--2); 
\draw[] (G--2) .. controls +(-0.5, 0.5) and +(0.5, 0.5) .. (G--1); 
\draw[] (G-1) .. controls +(0.5, -0.5) and +(-0.5, -0.5) .. (G-2); 
\draw[] (G-2) .. controls +(0.5, -0.5) and +(-0.5, -0.5) .. (G-3); 
\end{tikzpicture} \right]_{\wasylozenge} $

 \ 

 \ 

\noindent $\lambda = (3)$, $\mu = (2, 1)$: \ 

\noindent $ \left[ 
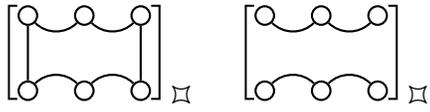
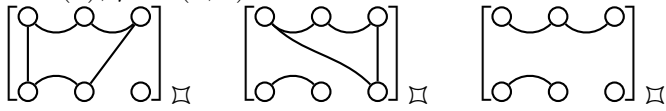
\begin{tikzpicture}[scale = 0.5,thick, baseline={(0,-1ex/2)}] 
\tikzstyle{vertex} = [shape = circle, minimum size = 7pt, inner sep = 1pt] 
\node[vertex] (G--3) at (3.0, -1) [shape = circle, draw] {}; 
\node[vertex] (G--2) at (1.5, -1) [shape = circle, draw] {}; 
\node[vertex] (G--1) at (0.0, -1) [shape = circle, draw] {}; 
\node[vertex] (G-1) at (0.0, 1) [shape = circle, draw] {}; 
\node[vertex] (G-2) at (1.5, 1) [shape = circle, draw] {}; 
\node[vertex] (G-3) at (3.0, 1) [shape = circle, draw] {}; 
\draw[] (G-1) .. controls +(0.5, -0.5) and +(-0.5, -0.5) .. (G-2); 
\draw[] (G-2) .. controls +(0.5, -0.5) and +(-0.5, -0.5) .. (G-3); 
\draw[] (G-3) .. controls +(-0.75, -1) and +(0.75, 1) .. (G--2); 
\draw[] (G--2) .. controls +(-0.5, 0.5) and +(0.5, 0.5) .. (G--1); 
\draw[] (G--1) .. controls +(0, 1) and +(0, -1) .. (G-1); 
\end{tikzpicture} \right]_{\wasylozenge} $
 \ \ \  
 $ \left[ \begin{tikzpicture}[scale = 0.5,thick, baseline={(0,-1ex/2)}] 
\tikzstyle{vertex} = [shape = circle, minimum size = 7pt, inner sep = 1pt] 
\node[vertex] (G--3) at (3.0, -1) [shape = circle, draw] {}; 
\node[vertex] (G-1) at (0.0, 1) [shape = circle, draw] {}; 
\node[vertex] (G-2) at (1.5, 1) [shape = circle, draw] {}; 
\node[vertex] (G-3) at (3.0, 1) [shape = circle, draw] {}; 
\node[vertex] (G--2) at (1.5, -1) [shape = circle, draw] {}; 
\node[vertex] (G--1) at (0.0, -1) [shape = circle, draw] {}; 
\draw[] (G-1) .. controls +(0.5, -0.5) and +(-0.5, -0.5) .. (G-2); 
\draw[] (G-2) .. controls +(0.5, -0.5) and +(-0.5, -0.5) .. (G-3); 
\draw[] (G-3) .. controls +(0, -1) and +(0, 1) .. (G--3); 
\draw[] (G--3) .. controls +(-1, 1) and +(1, -1) .. (G-1); 
\draw[] (G--2) .. controls +(-0.5, 0.5) and +(0.5, 0.5) .. (G--1); 
\end{tikzpicture} \right]_{\wasylozenge} $
 \ \ \  
 $ \left[ \begin{tikzpicture}[scale = 0.5,thick, baseline={(0,-1ex/2)}] 
\tikzstyle{vertex} = [shape = circle, minimum size = 7pt, inner sep = 1pt] 
\node[vertex] (G--3) at (3.0, -1) [shape = circle, draw] {}; 
\node[vertex] (G--2) at (1.5, -1) [shape = circle, draw] {}; 
\node[vertex] (G--1) at (0.0, -1) [shape = circle, draw] {}; 
\node[vertex] (G-1) at (0.0, 1) [shape = circle, draw] {}; 
\node[vertex] (G-2) at (1.5, 1) [shape = circle, draw] {}; 
\node[vertex] (G-3) at (3.0, 1) [shape = circle, draw] {}; 
\draw[] (G--2) .. controls +(-0.5, 0.5) and +(0.5, 0.5) .. (G--1); 
\draw[] (G-1) .. controls +(0.5, -0.5) and +(-0.5, -0.5) .. (G-2); 
\draw[] (G-2) .. controls +(0.5, -0.5) and +(-0.5, -0.5) .. (G-3); 
\end{tikzpicture} \right]_{\wasylozenge} $ 

 \ 

 \ 

\noindent $\lambda = (3)$, $\mu = (1, 1, 1)$: \ 

\noindent $ \left[ \begin{tikzpicture}[scale = 0.5,thick, baseline={(0,-1ex/2)}] 
\tikzstyle{vertex} = [shape = circle, minimum size = 7pt, inner sep = 1pt] 
\node[vertex] (G--3) at (3.0, -1) [shape = circle, draw] {}; 
\node[vertex] (G--2) at (1.5, -1) [shape = circle, draw] {}; 
\node[vertex] (G--1) at (0.0, -1) [shape = circle, draw] {}; 
\node[vertex] (G-1) at (0.0, 1) [shape = circle, draw] {}; 
\node[vertex] (G-2) at (1.5, 1) [shape = circle, draw] {}; 
\node[vertex] (G-3) at (3.0, 1) [shape = circle, draw] {}; 
\draw[] (G-1) .. controls +(0.5, -0.5) and +(-0.5, -0.5) .. (G-2); 
\draw[] (G-2) .. controls +(0.5, -0.5) and +(-0.5, -0.5) .. (G-3); 
\draw[] (G-3) .. controls +(-1, -1) and +(1, 1) .. (G--1); 
\draw[] (G--1) .. controls +(0, 1) and +(0, -1) .. (G-1); 
\end{tikzpicture} \right]_{\wasylozenge} $ 
 \ \ \  
 $ \left[ \begin{tikzpicture}[scale = 0.5,thick, baseline={(0,-1ex/2)}] 
\tikzstyle{vertex} = [shape = circle, minimum size = 7pt, inner sep = 1pt] 
\node[vertex] (G--3) at (3.0, -1) [shape = circle, draw] {}; 
\node[vertex] (G--2) at (1.5, -1) [shape = circle, draw] {}; 
\node[vertex] (G--1) at (0.0, -1) [shape = circle, draw] {}; 
\node[vertex] (G-1) at (0.0, 1) [shape = circle, draw] {}; 
\node[vertex] (G-2) at (1.5, 1) [shape = circle, draw] {}; 
\node[vertex] (G-3) at (3.0, 1) [shape = circle, draw] {}; 
\draw[] (G-1) .. controls +(0.5, -0.5) and +(-0.5, -0.5) .. (G-2); 
\draw[] (G-2) .. controls +(0.5, -0.5) and +(-0.5, -0.5) .. (G-3); 
\end{tikzpicture} \right]_{\wasylozenge} $ 

 \ 

 \ 

\noindent $\lambda = (2, 1)$, $\mu = (3)$: \ 

\noindent $ \left[ \begin{tikzpicture}[scale = 0.5,thick, baseline={(0,-1ex/2)}] 
\tikzstyle{vertex} = [shape = circle, minimum size = 7pt, inner sep = 1pt] 
\node[vertex] (G--3) at (3.0, -1) [shape = circle, draw] {}; 
\node[vertex] (G--2) at (1.5, -1) [shape = circle, draw] {}; 
\node[vertex] (G--1) at (0.0, -1) [shape = circle, draw] {}; 
\node[vertex] (G-1) at (0.0, 1) [shape = circle, draw] {}; 
\node[vertex] (G-2) at (1.5, 1) [shape = circle, draw] {}; 
\node[vertex] (G-3) at (3.0, 1) [shape = circle, draw] {}; 
\draw[] (G-1) .. controls +(0.5, -0.5) and +(-0.5, -0.5) .. (G-2); 
\draw[] (G-2) .. controls +(0.75, -1) and +(-0.75, 1) .. (G--3); 
\draw[] (G--3) .. controls +(-0.5, 0.5) and +(0.5, 0.5) .. (G--2); 
\draw[] (G--2) .. controls +(-0.5, 0.5) and +(0.5, 0.5) .. (G--1); 
\draw[] (G--1) .. controls +(0, 1) and +(0, -1) .. (G-1); 
\end{tikzpicture} \right]_{\wasylozenge} $ 
 \ \ \  
 $ \left[ \begin{tikzpicture}[scale = 0.5,thick, baseline={(0,-1ex/2)}] 
\tikzstyle{vertex} = [shape = circle, minimum size = 7pt, inner sep = 1pt] 
\node[vertex] (G--3) at (3.0, -1) [shape = circle, draw] {}; 
\node[vertex] (G--2) at (1.5, -1) [shape = circle, draw] {}; 
\node[vertex] (G--1) at (0.0, -1) [shape = circle, draw] {}; 
\node[vertex] (G-3) at (3.0, 1) [shape = circle, draw] {}; 
\node[vertex] (G-1) at (0.0, 1) [shape = circle, draw] {}; 
\node[vertex] (G-2) at (1.5, 1) [shape = circle, draw] {}; 
\draw[] (G-3) .. controls +(0, -1) and +(0, 1) .. (G--3); 
\draw[] (G--3) .. controls +(-0.5, 0.5) and +(0.5, 0.5) .. (G--2); 
\draw[] (G--2) .. controls +(-0.5, 0.5) and +(0.5, 0.5) .. (G--1); 
\draw[] (G--1) .. controls +(1, 1) and +(-1, -1) .. (G-3); 
\draw[] (G-1) .. controls +(0.5, -0.5) and +(-0.5, -0.5) .. (G-2); 
\end{tikzpicture} \right]_{\wasylozenge} $
 \ \ \  
 $ \left[ \begin{tikzpicture}[scale = 0.5,thick, baseline={(0,-1ex/2)}] 
\tikzstyle{vertex} = [shape = circle, minimum size = 7pt, inner sep = 1pt] 
\node[vertex] (G--3) at (3.0, -1) [shape = circle, draw] {}; 
\node[vertex] (G--2) at (1.5, -1) [shape = circle, draw] {}; 
\node[vertex] (G--1) at (0.0, -1) [shape = circle, draw] {}; 
\node[vertex] (G-1) at (0.0, 1) [shape = circle, draw] {}; 
\node[vertex] (G-2) at (1.5, 1) [shape = circle, draw] {}; 
\node[vertex] (G-3) at (3.0, 1) [shape = circle, draw] {}; 
\draw[] (G--3) .. controls +(-0.5, 0.5) and +(0.5, 0.5) .. (G--2); 
\draw[] (G--2) .. controls +(-0.5, 0.5) and +(0.5, 0.5) .. (G--1); 
\draw[] (G-1) .. controls +(0.5, -0.5) and +(-0.5, -0.5) .. (G-2); 
\end{tikzpicture} \right]_{\wasylozenge} $ 

 \ 

 \ 

\noindent $\lambda = (2, 1)$, $\mu = (2, 1)$: \ 

\noindent $ \left[ \begin{tikzpicture}[scale = 0.5,thick, baseline={(0,-1ex/2)}] 
\tikzstyle{vertex} = [shape = circle, minimum size = 7pt, inner sep = 1pt] 
\node[vertex] (G--3) at (3.0, -1) [shape = circle, draw] {}; 
\node[vertex] (G-1) at (0.0, 1) [shape = circle, draw] {}; 
\node[vertex] (G-2) at (1.5, 1) [shape = circle, draw] {}; 
\node[vertex] (G--2) at (1.5, -1) [shape = circle, draw] {}; 
\node[vertex] (G--1) at (0.0, -1) [shape = circle, draw] {}; 
\node[vertex] (G-3) at (3.0, 1) [shape = circle, draw] {}; 
\draw[] (G-1) .. controls +(0.5, -0.5) and +(-0.5, -0.5) .. (G-2); 
\draw[] (G-2) .. controls +(0.75, -1) and +(-0.75, 1) .. (G--3); 
\draw[] (G--3) .. controls +(-1, 1) and +(1, -1) .. (G-1); 
\draw[] (G-3) .. controls +(-0.75, -1) and +(0.75, 1) .. (G--2); 
\draw[] (G--2) .. controls +(-0.5, 0.5) and +(0.5, 0.5) .. (G--1); 
\draw[] (G--1) .. controls +(1, 1) and +(-1, -1) .. (G-3); 
\end{tikzpicture} \right]_{\wasylozenge} $
 \ \ \  
 $ \left[ \begin{tikzpicture}[scale = 0.5,thick, baseline={(0,-1ex/2)}] 
\tikzstyle{vertex} = [shape = circle, minimum size = 7pt, inner sep = 1pt] 
\node[vertex] (G--3) at (3.0, -1) [shape = circle, draw] {}; 
\node[vertex] (G-3) at (3.0, 1) [shape = circle, draw] {}; 
\node[vertex] (G--2) at (1.5, -1) [shape = circle, draw] {}; 
\node[vertex] (G--1) at (0.0, -1) [shape = circle, draw] {}; 
\node[vertex] (G-1) at (0.0, 1) [shape = circle, draw] {}; 
\node[vertex] (G-2) at (1.5, 1) [shape = circle, draw] {}; 
\draw[] (G-3) .. controls +(0, -1) and +(0, 1) .. (G--3); 
\draw[] (G-1) .. controls +(0.5, -0.5) and +(-0.5, -0.5) .. (G-2); 
\draw[] (G-2) .. controls +(0, -1) and +(0, 1) .. (G--2); 
\draw[] (G--2) .. controls +(-0.5, 0.5) and +(0.5, 0.5) .. (G--1); 
\draw[] (G--1) .. controls +(0, 1) and +(0, -1) .. (G-1); 
\end{tikzpicture} \right]_{\wasylozenge} $ 
 \ \ \  
 $ \left[ \begin{tikzpicture}[scale = 0.5,thick, baseline={(0,-1ex/2)}] 
\tikzstyle{vertex} = [shape = circle, minimum size = 7pt, inner sep = 1pt] 
\node[vertex] (G--3) at (3.0, -1) [shape = circle, draw] {}; 
\node[vertex] (G--2) at (1.5, -1) [shape = circle, draw] {}; 
\node[vertex] (G--1) at (0.0, -1) [shape = circle, draw] {}; 
\node[vertex] (G-3) at (3.0, 1) [shape = circle, draw] {}; 
\node[vertex] (G-1) at (0.0, 1) [shape = circle, draw] {}; 
\node[vertex] (G-2) at (1.5, 1) [shape = circle, draw] {}; 
\draw[] (G-3) .. controls +(-0.75, -1) and +(0.75, 1) .. (G--2); 
\draw[] (G--2) .. controls +(-0.5, 0.5) and +(0.5, 0.5) .. (G--1); 
\draw[] (G--1) .. controls +(1, 1) and +(-1, -1) .. (G-3); 
\draw[] (G-1) .. controls +(0.5, -0.5) and +(-0.5, -0.5) .. (G-2); 
\end{tikzpicture} \right]_{\wasylozenge} $ 
 \ \ \  
 $ \left[ \begin{tikzpicture}[scale = 0.5,thick, baseline={(0,-1ex/2)}] 
\tikzstyle{vertex} = [shape = circle, minimum size = 7pt, inner sep = 1pt] 
\node[vertex] (G--3) at (3.0, -1) [shape = circle, draw] {}; 
\node[vertex] (G-3) at (3.0, 1) [shape = circle, draw] {}; 
\node[vertex] (G--2) at (1.5, -1) [shape = circle, draw] {}; 
\node[vertex] (G--1) at (0.0, -1) [shape = circle, draw] {}; 
\node[vertex] (G-1) at (0.0, 1) [shape = circle, draw] {}; 
\node[vertex] (G-2) at (1.5, 1) [shape = circle, draw] {}; 
\draw[] (G-3) .. controls +(0, -1) and +(0, 1) .. (G--3); 
\draw[] (G--2) .. controls +(-0.5, 0.5) and +(0.5, 0.5) .. (G--1); 
\draw[] (G-1) .. controls +(0.5, -0.5) and +(-0.5, -0.5) .. (G-2); 
\end{tikzpicture} \right]_{\wasylozenge} $ 

 \ 

 \ 

\noindent $ \left[ \begin{tikzpicture}[scale = 0.5,thick, baseline={(0,-1ex/2)}] 
\tikzstyle{vertex} = [shape = circle, minimum size = 7pt, inner sep = 1pt] 
\node[vertex] (G--3) at (3.0, -1) [shape = circle, draw] {}; 
\node[vertex] (G--2) at (1.5, -1) [shape = circle, draw] {}; 
\node[vertex] (G--1) at (0.0, -1) [shape = circle, draw] {}; 
\node[vertex] (G-1) at (0.0, 1) [shape = circle, draw] {}; 
\node[vertex] (G-2) at (1.5, 1) [shape = circle, draw] {}; 
\node[vertex] (G-3) at (3.0, 1) [shape = circle, draw] {}; 
\draw[] (G-1) .. controls +(0.5, -0.5) and +(-0.5, -0.5) .. (G-2); 
\draw[] (G-2) .. controls +(0, -1) and +(0, 1) .. (G--2); 
\draw[] (G--2) .. controls +(-0.5, 0.5) and +(0.5, 0.5) .. (G--1); 
\draw[] (G--1) .. controls +(0, 1) and +(0, -1) .. (G-1); 
\end{tikzpicture} \right]_{\wasylozenge} $ 
 \ \ \  
 $ \left[ \begin{tikzpicture}[scale = 0.5,thick, baseline={(0,-1ex/2)}] 
\tikzstyle{vertex} = [shape = circle, minimum size = 7pt, inner sep = 1pt] 
\node[vertex] (G--3) at (3.0, -1) [shape = circle, draw] {}; 
\node[vertex] (G-1) at (0.0, 1) [shape = circle, draw] {}; 
\node[vertex] (G-2) at (1.5, 1) [shape = circle, draw] {}; 
\node[vertex] (G--2) at (1.5, -1) [shape = circle, draw] {}; 
\node[vertex] (G--1) at (0.0, -1) [shape = circle, draw] {}; 
\node[vertex] (G-3) at (3.0, 1) [shape = circle, draw] {}; 
\draw[] (G-1) .. controls +(0.5, -0.5) and +(-0.5, -0.5) .. (G-2); 
\draw[] (G-2) .. controls +(0.75, -1) and +(-0.75, 1) .. (G--3); 
\draw[] (G--3) .. controls +(-1, 1) and +(1, -1) .. (G-1); 
\draw[] (G--2) .. controls +(-0.5, 0.5) and +(0.5, 0.5) .. (G--1); 
\end{tikzpicture} \right]_{\wasylozenge} $ 
 \ \ \  
 $ \left[ \begin{tikzpicture}[scale = 0.5,thick, baseline={(0,-1ex/2)}] 
\tikzstyle{vertex} = [shape = circle, minimum size = 7pt, inner sep = 1pt] 
\node[vertex] (G--3) at (3.0, -1) [shape = circle, draw] {}; 
\node[vertex] (G--2) at (1.5, -1) [shape = circle, draw] {}; 
\node[vertex] (G--1) at (0.0, -1) [shape = circle, draw] {}; 
\node[vertex] (G-1) at (0.0, 1) [shape = circle, draw] {}; 
\node[vertex] (G-2) at (1.5, 1) [shape = circle, draw] {}; 
\node[vertex] (G-3) at (3.0, 1) [shape = circle, draw] {}; 
\draw[] (G--2) .. controls +(-0.5, 0.5) and +(0.5, 0.5) .. (G--1); 
\draw[] (G-1) .. controls +(0.5, -0.5) and +(-0.5, -0.5) .. (G-2); 
\end{tikzpicture} \right]_{\wasylozenge} $ 

 \ 

 \ 

\noindent $\lambda = (2, 1)$, $\mu = (1, 1, 1)$: \ 

\noindent $ \left[ \begin{tikzpicture}[scale = 0.5,thick, baseline={(0,-1ex/2)}] 
\tikzstyle{vertex} = [shape = circle, minimum size = 7pt, inner sep = 1pt] 
\node[vertex] (G--3) at (3.0, -1) [shape = circle, draw] {}; 
\node[vertex] (G-3) at (3.0, 1) [shape = circle, draw] {}; 
\node[vertex] (G--2) at (1.5, -1) [shape = circle, draw] {}; 
\node[vertex] (G--1) at (0.0, -1) [shape = circle, draw] {}; 
\node[vertex] (G-1) at (0.0, 1) [shape = circle, draw] {}; 
\node[vertex] (G-2) at (1.5, 1) [shape = circle, draw] {}; 
\draw[] (G-3) .. controls +(0, -1) and +(0, 1) .. (G--3); 
\draw[] (G-1) .. controls +(0.5, -0.5) and +(-0.5, -0.5) .. (G-2); 
\end{tikzpicture} \right]_{\wasylozenge} $ 
 \ \ \  
 $ \left[ \begin{tikzpicture}[scale = 0.5,thick, baseline={(0,-1ex/2)}] 
\tikzstyle{vertex} = [shape = circle, minimum size = 7pt, inner sep = 1pt] 
\node[vertex] (G--3) at (3.0, -1) [shape = circle, draw] {}; 
\node[vertex] (G--2) at (1.5, -1) [shape = circle, draw] {}; 
\node[vertex] (G--1) at (0.0, -1) [shape = circle, draw] {}; 
\node[vertex] (G-1) at (0.0, 1) [shape = circle, draw] {}; 
\node[vertex] (G-2) at (1.5, 1) [shape = circle, draw] {}; 
\node[vertex] (G-3) at (3.0, 1) [shape = circle, draw] {}; 
\draw[] (G-1) .. controls +(0.5, -0.5) and +(-0.5, -0.5) .. (G-2); 
\draw[] (G-2) .. controls +(-0.75, -1) and +(0.75, 1) .. (G--1); 
\draw[] (G--1) .. controls +(0, 1) and +(0, -1) .. (G-1); 
\end{tikzpicture} \right]_{\wasylozenge} $ 
 \ \ \  
 $ \left[ \begin{tikzpicture}[scale = 0.5,thick, baseline={(0,-1ex/2)}] 
\tikzstyle{vertex} = [shape = circle, minimum size = 7pt, inner sep = 1pt] 
\node[vertex] (G--3) at (3.0, -1) [shape = circle, draw] {}; 
\node[vertex] (G--2) at (1.5, -1) [shape = circle, draw] {}; 
\node[vertex] (G--1) at (0.0, -1) [shape = circle, draw] {}; 
\node[vertex] (G-1) at (0.0, 1) [shape = circle, draw] {}; 
\node[vertex] (G-2) at (1.5, 1) [shape = circle, draw] {}; 
\node[vertex] (G-3) at (3.0, 1) [shape = circle, draw] {}; 
\draw[] (G-1) .. controls +(0.5, -0.5) and +(-0.5, -0.5) .. (G-2); 
\end{tikzpicture} \right]_{\wasylozenge} $ 

 \ 

 \ 

\noindent $\lambda = (1, 1, 1)$, $\mu = (3)$:  \  

\noindent $ \left[ \begin{tikzpicture}[scale = 0.5,thick, baseline={(0,-1ex/2)}] 
\tikzstyle{vertex} = [shape = circle, minimum size = 7pt, inner sep = 1pt] 
\node[vertex] (G--3) at (3.0, -1) [shape = circle, draw] {}; 
\node[vertex] (G--2) at (1.5, -1) [shape = circle, draw] {}; 
\node[vertex] (G--1) at (0.0, -1) [shape = circle, draw] {}; 
\node[vertex] (G-1) at (0.0, 1) [shape = circle, draw] {}; 
\node[vertex] (G-2) at (1.5, 1) [shape = circle, draw] {}; 
\node[vertex] (G-3) at (3.0, 1) [shape = circle, draw] {}; 
\draw[] (G-1) .. controls +(1, -1) and +(-1, 1) .. (G--3); 
\draw[] (G--3) .. controls +(-0.5, 0.5) and +(0.5, 0.5) .. (G--2); 
\draw[] (G--2) .. controls +(-0.5, 0.5) and +(0.5, 0.5) .. (G--1); 
\draw[] (G--1) .. controls +(0, 1) and +(0, -1) .. (G-1); 
\end{tikzpicture} \right]_{\wasylozenge} $ 
 \ \ \  
 $ \left[ \begin{tikzpicture}[scale = 0.5,thick, baseline={(0,-1ex/2)}] 
\tikzstyle{vertex} = [shape = circle, minimum size = 7pt, inner sep = 1pt] 
\node[vertex] (G--3) at (3.0, -1) [shape = circle, draw] {}; 
\node[vertex] (G--2) at (1.5, -1) [shape = circle, draw] {}; 
\node[vertex] (G--1) at (0.0, -1) [shape = circle, draw] {}; 
\node[vertex] (G-1) at (0.0, 1) [shape = circle, draw] {}; 
\node[vertex] (G-2) at (1.5, 1) [shape = circle, draw] {}; 
\node[vertex] (G-3) at (3.0, 1) [shape = circle, draw] {}; 
\draw[] (G--3) .. controls +(-0.5, 0.5) and +(0.5, 0.5) .. (G--2); 
\draw[] (G--2) .. controls +(-0.5, 0.5) and +(0.5, 0.5) .. (G--1); 
\end{tikzpicture} \right]_{\wasylozenge} $ 

 \ 

 \ 

\noindent $\lambda = (1, 1, 1)$, $\mu = (2, 1)$:  \ 

\noindent $ \left[ \begin{tikzpicture}[scale = 0.5,thick, baseline={(0,-1ex/2)}] 
\tikzstyle{vertex} = [shape = circle, minimum size = 7pt, inner sep = 1pt] 
\node[vertex] (G--3) at (3.0, -1) [shape = circle, draw] {}; 
\node[vertex] (G-3) at (3.0, 1) [shape = circle, draw] {}; 
\node[vertex] (G--2) at (1.5, -1) [shape = circle, draw] {}; 
\node[vertex] (G--1) at (0.0, -1) [shape = circle, draw] {}; 
\node[vertex] (G-1) at (0.0, 1) [shape = circle, draw] {}; 
\node[vertex] (G-2) at (1.5, 1) [shape = circle, draw] {}; 
\draw[] (G-3) .. controls +(0, -1) and +(0, 1) .. (G--3); 
\draw[] (G--2) .. controls +(-0.5, 0.5) and +(0.5, 0.5) .. (G--1); 
\end{tikzpicture} \right]_{\wasylozenge} $ 
 \ \ \  
 $ \left[ \begin{tikzpicture}[scale = 0.5,thick, baseline={(0,-1ex/2)}] 
\tikzstyle{vertex} = [shape = circle, minimum size = 7pt, inner sep = 1pt] 
\node[vertex] (G--3) at (3.0, -1) [shape = circle, draw] {}; 
\node[vertex] (G--2) at (1.5, -1) [shape = circle, draw] {}; 
\node[vertex] (G--1) at (0.0, -1) [shape = circle, draw] {}; 
\node[vertex] (G-1) at (0.0, 1) [shape = circle, draw] {}; 
\node[vertex] (G-2) at (1.5, 1) [shape = circle, draw] {}; 
\node[vertex] (G-3) at (3.0, 1) [shape = circle, draw] {}; 
\draw[] (G-1) .. controls +(0.75, -1) and +(-0.75, 1) .. (G--2); 
\draw[] (G--2) .. controls +(-0.5, 0.5) and +(0.5, 0.5) .. (G--1); 
\draw[] (G--1) .. controls +(0, 1) and +(0, -1) .. (G-1); 
\end{tikzpicture} \right]_{\wasylozenge} $ 
 \ \ \  
 $ \left[ \begin{tikzpicture}[scale = 0.5,thick, baseline={(0,-1ex/2)}] 
\tikzstyle{vertex} = [shape = circle, minimum size = 7pt, inner sep = 1pt] 
\node[vertex] (G--3) at (3.0, -1) [shape = circle, draw] {}; 
\node[vertex] (G--2) at (1.5, -1) [shape = circle, draw] {}; 
\node[vertex] (G--1) at (0.0, -1) [shape = circle, draw] {}; 
\node[vertex] (G-1) at (0.0, 1) [shape = circle, draw] {}; 
\node[vertex] (G-2) at (1.5, 1) [shape = circle, draw] {}; 
\node[vertex] (G-3) at (3.0, 1) [shape = circle, draw] {}; 
\draw[] (G--2) .. controls +(-0.5, 0.5) and +(0.5, 0.5) .. (G--1); 
\end{tikzpicture} \right]_{\wasylozenge} $

 \ 

 \ 

\noindent $\lambda = (1, 1, 1)$, $\mu = (1, 1, 1)$: \ 

\noindent $ \left[ \begin{tikzpicture}[scale = 0.5,thick, baseline={(0,-1ex/2)}] 
\tikzstyle{vertex} = [shape = circle, minimum size = 7pt, inner sep = 1pt] 
\node[vertex] (G--3) at (3.0, -1) [shape = circle, draw] {}; 
\node[vertex] (G-3) at (3.0, 1) [shape = circle, draw] {}; 
\node[vertex] (G--2) at (1.5, -1) [shape = circle, draw] {}; 
\node[vertex] (G-2) at (1.5, 1) [shape = circle, draw] {}; 
\node[vertex] (G--1) at (0.0, -1) [shape = circle, draw] {}; 
\node[vertex] (G-1) at (0.0, 1) [shape = circle, draw] {}; 
\draw[] (G-3) .. controls +(0, -1) and +(0, 1) .. (G--3); 
\draw[] (G-2) .. controls +(0, -1) and +(0, 1) .. (G--2); 
\draw[] (G-1) .. controls +(0, -1) and +(0, 1) .. (G--1); 
\end{tikzpicture} \right]_{\wasylozenge} $ 
 \ \ \  
 $ \left[ \begin{tikzpicture}[scale = 0.5,thick, baseline={(0,-1ex/2)}] 
\tikzstyle{vertex} = [shape = circle, minimum size = 7pt, inner sep = 1pt] 
\node[vertex] (G--3) at (3.0, -1) [shape = circle, draw] {}; 
\node[vertex] (G--2) at (1.5, -1) [shape = circle, draw] {}; 
\node[vertex] (G-2) at (1.5, 1) [shape = circle, draw] {}; 
\node[vertex] (G--1) at (0.0, -1) [shape = circle, draw] {}; 
\node[vertex] (G-1) at (0.0, 1) [shape = circle, draw] {}; 
\node[vertex] (G-3) at (3.0, 1) [shape = circle, draw] {}; 
\draw[] (G-2) .. controls +(0, -1) and +(0, 1) .. (G--2); 
\draw[] (G-1) .. controls +(0, -1) and +(0, 1) .. (G--1); 
\end{tikzpicture} \right]_{\wasylozenge} $
 \ \ \  
 $ \left[ \begin{tikzpicture}[scale = 0.5,thick, baseline={(0,-1ex/2)}] 
\tikzstyle{vertex} = [shape = circle, minimum size = 7pt, inner sep = 1pt] 
\node[vertex] (G--3) at (3.0, -1) [shape = circle, draw] {}; 
\node[vertex] (G--2) at (1.5, -1) [shape = circle, draw] {}; 
\node[vertex] (G--1) at (0.0, -1) [shape = circle, draw] {}; 
\node[vertex] (G-1) at (0.0, 1) [shape = circle, draw] {}; 
\node[vertex] (G-2) at (1.5, 1) [shape = circle, draw] {}; 
\node[vertex] (G-3) at (3.0, 1) [shape = circle, draw] {}; 
\draw[] (G-1) .. controls +(0, -1) and +(0, 1) .. (G--1); 
\end{tikzpicture} \right]_{\wasylozenge} $
 \ \ \  
 $ \left[ \begin{tikzpicture}[scale = 0.5,thick, baseline={(0,-1ex/2)}] 
\tikzstyle{vertex} = [shape = circle, minimum size = 7pt, inner sep = 1pt] 
\node[vertex] (G--3) at (3.0, -1) [shape = circle, draw] {}; 
\node[vertex] (G--2) at (1.5, -1) [shape = circle, draw] {}; 
\node[vertex] (G--1) at (0.0, -1) [shape = circle, draw] {}; 
\node[vertex] (G-1) at (0.0, 1) [shape = circle, draw] {}; 
\node[vertex] (G-2) at (1.5, 1) [shape = circle, draw] {}; 
\node[vertex] (G-3) at (3.0, 1) [shape = circle, draw] {}; 
\end{tikzpicture} \right]_{\wasylozenge} $ 
\end{example}

 An advantage of Theorem \ref{dimensiontheorem} is given by how it provides a computationally practical and efficient way of computing the dimensions 
 for centralizer algebras of the form $\text{End}_{S_{n}}\big( V^{\boxtimes k} \big)$. Letting $2k \leq n$, it may be verified that 
\begin{equation}\label{notoeis}
 \left( \text{dim}\left( \text{End}_{S_{n}}\big( V^{\boxtimes k} \big) \right) : k \in \mathbb{N} 
 \right) = (2, 9, 29, 94, 275, 768, 2055, \ldots ), 
\end{equation}
 noting that the $\text{End}_{S_{n}}\big( V^{\boxtimes 2} \big) = 9$ case is illustrated in 
 Example \ref{dimension9example}. 
 The integer sequence in \eqref{notoeis}
 is not currently included in the On-line Encyclopedia of Integer Sequences, 
 and this suggests that our centralizer algebra construction is new. 

\section{Conclusion} 
 From our construction of the basis for $\text{End}_{S_{n}}\big( V^{\boxtimes k} \big)$ highlighted in Theorem \ref{Cbasis}, we may devise an 
 analogue of the orbit basis multiplication rule for $\text{End}_{S_{n}}\big( V^{\otimes k} \big)$, following a similar approach as in the work of 
 Benkart and Halverson \cite{BenkartHalverson2019}. This, in turn, can be used to obtain a diagram-like basis for $\text{End}_{S_{n}}\big( V^{\boxtimes k} \big)$. 
 This motivates a full exploration of the representation theory of $\text{End}_{S_{n}}\big( V^{\boxtimes k} \big)$ 
 and its relation to the representation theory for partition algebras. 

\subsection*{Acknowledgements}
 The author was supported through a Killam Postdoctoral Fellowship from the Killam Trusts.

\bibliographystyle{plain}
\bibliography{seprefe}

\end{document}